\documentclass[oneside,reqno,11pt]{amsart}

\usepackage{graphicx}
\usepackage{epstopdf}
\usepackage{amsmath}
\usepackage{amssymb}
\usepackage[top=1in, bottom=1.25in, left=1.25in, right=1.25in]{geometry}
\usepackage{textcomp}
\usepackage{mathrsfs}
\usepackage{tikz}
\usepackage{tikz-cd}

\usepackage[utf8]{inputenc}
\usepackage{mathabx}
\usepackage{bm}
\usepackage{amsthm}
\usepackage{amsfonts}
\usepackage[colorlinks=true,linkcolor=blue]{hyperref}
\usepackage{pb-diagram}
\usepackage{epstopdf}
\usepackage{lipsum}
\usepackage{secdot}
\usepackage{color}
\usepackage[normalem]{ulem}
\usepackage{enumitem} 

\usepackage{dsfont}
\usepackage{appendix}
\usepackage{color}
\usepackage{calc}
\usepackage{slantsc}
\usepackage[sc]{mathpazo}
\usepackage{indentfirst}
\usepackage{mathtools}

\theoremstyle{plain}
\newtheorem{prop}{Proposition}[section]
\newtheorem{theorem}[prop]{Theorem}

\newtheorem{corollary}[prop]{Corollary}
\newtheorem{definition}[prop]{Definition}
\newtheorem{example}[prop]{Example}
\newtheorem{lemma}[prop]{Lemma}
\newtheorem{remark}[prop]{Remark}

\numberwithin{equation}{section}

\newcommand*{\medcap}{\mathbin{\scalebox{1.5}{\ensuremath{\cap}}}}

\newcommand*{\medcup}{\mathbin{\scalebox{1.5}{\ensuremath{\cup}}}}

\newcommand{\twobar}{/\kern-0.2em/}
\newcommand{\threebar}{/\kern-0.2em/\kern-0.2em/}

\newcommand\norm[1]{\left\lVert#1\right\rVert}

\newcommand{\bP}{\mathbb{P}}
\newcommand{\bS}{\mathbb{S}}
\newcommand{\Z}{\mathbb{Z}}

\newcommand{\R}{\mathbb{R}}
\newcommand{\C}{\mathbb{C}}
\newcommand{\T}{\mathbb{T}}
\newcommand{\CP}{\mathbb{P}}
\newcommand{\bu}{\mathbf{u}}

\newcommand{\Hom}{\mathrm{Hom}}

\newcommand{\Kernel}{\mathrm{Ker}}

\newcommand{\U}{\mathrm{U}}

\newcommand{\fm}{\mathfrak{m}}

\newcommand{\Spec}{\operatorname{Spec}}

\begin{document}
	\title{SYZ mirror symmetry for hypertoric varieties}
	\author[Lau]{Siu-Cheong Lau}
	\address{Department of Mathematics and Statistics\\ Boston University}
	\email{lau@math.bu.edu}
		\author[Zheng]{Xiao Zheng}
	\address{Department of Mathematics and Statistics\\ Boston University}
	\email{xiaoz259@bu.edu}
	
	\begin{abstract}
		We construct a Lagrangian torus fibration on a smooth hypertoric variety and a corresponding SYZ mirror variety using $T$-duality and generating functions of open Gromov--Witten invariants.  The variety is singular in general. We construct a resolution of the variety using the wall and chamber structure on the base of the SYZ fibration.
	\end{abstract}
	
	\maketitle
	
	\section{Introduction}
	
	Mirror symmetry has made powerful and striking predictions in enumerative geometry.  It has led to groundbreaking results in algebraic and differential geometry, number theory, gauge theory and other branches of mathematics.  
	
	 Strominger-Yau-Zaslow \cite{SYZ} proposed that mirror symmetry can be understood as torus duality.  It conjectured a geometric construction of mirror manifolds and a canonical transformation to derive the homological mirror symmetry conjecture \cite{Kont-HMS}.  
	
	There have been a lot of breakthroughs in SYZ mirror symmetry.  The Gross-Siebert program \cite{GS07} gave a purely algebraic method to reconstruct the mirror manifolds.  Auroux \cite{Auroux07,Auroux09} provided a symplectic approach to SYZ and the Gross-Siebert program.  Moreover, Floer theory of wall-crossing was developed in Pascaleff-Tonkonog
	\cite{PT17} based on the work of Seidel \cite{Seidel-lect}.  Furthermore, based on the works of Fukaya-Oh-Ohta-Ono \cite{FOOO,FOOO-T}, Seidel \cite{Seidel-g2} and Akaho-Joyce \cite{AJ},  deformation and moduli theory of Lagrangian immersions are being developed by Cho-Hong-Lau \cite{CHL,CHL2,HL} which enhance and generalize the SYZ program.  Floer theory of generic singular SYZ fibers and its relation with wall-crossing were understood in \cite{HKL18,ERT}.  Finally, the family Floer theory initiated by Fukaya \cite{Fuk-famFl} and further developed by Tu \cite{Tu-reconstruction,Tu-FM} and Abouzaid \cite{Ab-famFl1,Ab-famFl2} provides a canonical functor which realizes the SYZ mirror transformation.  
	
	In view of these recent developments, SYZ mirror symmetry can be understood via a local-to-global approach.  First we need to understand SYZ transformation for local geometries around singular Lagrangians.  Second we need to glue the local mirrors using Floer-theoretical methods.  
	
	Toric Calabi-Yau manifolds and their mirrors provide a rich source of local models.  Wall-crossing and SYZ mirror construction have been understood due to the works of Auroux \cite{Auroux07, Auroux09}, Chan-Lau-Leung \cite{CLL}, Abouzaid-Auroux-Katzarkov \cite{AAK} and Chan-Cho-Lau-Tseng \cite{CCLT2}.  Using the local models, geometric transitions have been studied by Castano-Bernard and Matessi \cite{CBM3} and other groups \cite{L13,CPU,KL1,KL2,L18}.
	
	In this paper we study SYZ for the hyper-Kähler analog of toric manifolds.  Analogous to toric manifolds, they are obtained as hyper-Kähler quotients of $T^*\C^n$.  Typical examples of hypertoric manifolds include $T^*\C\bP^n$ and crepant resolutions of $A_n$ singularities.  We expect that they should provide useful local models to understand mirror symmetry for holomorphic symplectic manifolds.
	
	The structure of the paper is as follows. In Section \ref{review}, we review the definition of properties of hypertoric varieties. We construct Lagrangian fibrations on hypertoric manifolds in Section \ref{sec:fib}.  It uses the techniques of Gross \cite{Gross-eg} and Goldstein \cite{Goldstein} by symplectic reduction, and Abouzaid-Auroux-Katzarkov \cite{AAK} by Moser argument.  The Lagrangian fibrations have codimension-one amoeba-like discriminant loci. 
	
	We carry out the SYZ mirror construction for hypertoric varieties in Section \ref{SYZ} with a brief review of SYZ in Section \ref{sec:SYZ}. We first analyze the walls over which the Lagrangian torus fibers bound holomorphic discs of Maslov index $0$ (Section \ref{sec:wall}). The walls divide the base of a Lagrangian fibration into chambers (Section \ref{sec:chambers}). We then find all the holomorphic discs of Maslov index $2$ bounded by a fiber in each chamber (Section \ref{sec:disc2}) and show their regularity (Section \ref{regularity}).  As a result we obtain the generating functions of open Gromov--Witten invariants which are countings of these holomorphic discs (Section \ref{sec:GF}). We compactify the manifold in order to have sufficiently many boundary divisors (Section \ref{sec:cptfy}).
	
	Finally, in Section \ref{sec:mirror}, we construct a  SYZ mirror variety as the spectrum of the ring of generating functions associated to boundary divisors. By construction the mirror we obtain is affine, and is singular in general. It should be viewed as the affinization of a smooth mirror.  A resolution is necessary to better understand the geometry.   We glue together a resolution using local charts coming from the wall and chamber structure of the SYZ base.  The gluing can be explained using Floer-theoretical techniques as in \cite{Seidel-lect,PT17,HL}, but we will leave this in future work. The variety admits another resolution by a multiplicative hypertoric variety (Section \ref{sec:multiplicative}).  In general these resolutions are topologically different.  We conclude with the following theorems.
	
	\begin{theorem}
Let $\mathfrak{M}_{u,\lambda}$ be a smooth hypertoric variety, and $D^-\subset\mathfrak{M}_{u,\lambda}$ a certain anti-canonical divisor (given by Equation (\ref{D-})). The SYZ mirror $\mathfrak{M}_{u,\lambda}^{\vee}$ of the pair $(\mathfrak{M}_{u,\lambda},D^-)$ is the affine variety
\[
\mathfrak{M}_{u,\lambda}^{\vee}=\left\{((\bm{u}_1,\bm{v}_1,\ldots,\bm{u}_d,\bm{v}_d), (\bm{Z}_{1},\ldots,\bm{Z}_{d}))\in\C^{2d}\times (\C^{\times})^d \mid \bm{u}_i\bm{v}_i=\prod_{k\in \bm{j}}(1+\bm{Z}_k), i=1,\ldots,d\right\},
\]
which admits a canonical resolution given by the wall and chamber structure of the SYZ base. 
	\end{theorem}

The notations are explained in Section \ref{sec:GF}.	
		
	\begin{theorem}
		Let $\mathfrak{M}$ be a smooth hypertoric variety which is obtained as a hyper-Kähler quotient of $T^*\C^n$ by a sub-torus $K\subset T^n$.  Its SYZ mirror is birational to the multiplicative hypertoric variety
		$\bm{\mu}^{-1}(q)//_{\chi} K_{\C}$
		where $\bm{\mu}$ is the multiplicative moment map, and $q\in K_\C$ is determined by the K\"ahler parameters of $\mathfrak{M}$, and $\chi\in\Hom(K_\C,\C^\times)$ is a generic character.
	\end{theorem}
	
	Below we introduce some important related works and questions that we wish to understand in the future.
	
	Closed-string equivariant mirror symmetry for hypertoric manifolds was found by Mcbreen and Shenfeld \cite{MS}.  They derived a presentation of the $T^d\times\C^\times$-equivariant quantum cohomology of a hypertoric manifold and relate it with the Gauss-Manin connection of the mirror moduli.  To understand the equivariant quantum cohomology from the SYZ perspective in this paper, we need to study equivariant Floer theory.
	
	In a recent preprint \cite{MW18}, Mcbreen and Webster showed that a category of equivariant coherent sheaves on a hypertoric variety are derived equivalent to the category of DQ-modules on the corresponding \textit{Dolbeault hypertoric variety}, establishing a version of homological mirror symmetry in the reverse direction.  Dolbeault hypertoric varieties as defined in \cite{MW18} are analog of hyper-Kähler quotient of Ooguri-Vafa space and carry canonical special Lagrangian torus fibrations.  
	
	In a subsequent work \cite{GMW}, Gammage, Mcbreen and Webster proved homological mirror symmetry for multiplicative hypertoric varieties.  Moreover, they conjectured that multiplicative hypertoric varieties are complements of additive hypertoric varieties $\mathfrak{M}_{u,\lambda}$ of some anti-canonical divisors (Conjecture 1.7 of \cite{GMW}).  It is an interesting direction to understand the relation with the anti-canonical divisor $D^-$ used in this paper, and mirror transformation of objects from the SYZ perspective.
	
	Furthermore, we believe hypertoric varieties are useful to understand mirror symmetry for cotangent bundles of smooth flag varieties.  Toric degenerations of flag varieties were used to construct their mirrors by Nishinou-Nohara-Ueda \cite{NNU,NU}.  It is reasonable to expect that mirrors of the total spaces of cotangent bundles of flag varieties are closely related to the mirrors of (singular) hypertoric varieties.
		
	\subsection*{Acknowledgment}
	The first named author is grateful to Conan Leung for bringing his interest to mirror symmetry for hypertoric varieties.  The authors thank to Yoosik Kim and Hansol Hong for useful discussions.  The work of the first named author is partially supported by the Simons collaboration grant. 
	
	\section{Review of hypertoric varieties}
	\label{review}
	In this section, we review the definition and basic properties of hypertoric varieties. We refer to \cite{BD,HS,Proudfoot} for more detailed account of the subject. All material in this section, except Proposition \ref{prop:complement}, are from the existing literature.
	
	\subsection{Hypertoric varieties} \label{sec:hypertoric}

Let $\mathfrak{t}^n$ and $\mathfrak{t}^d$ be real vector spaces of dimension $n$ and $d$, respectively. Let $\mathfrak{t}^n_{\Z}\subset\mathfrak{t}^n$ and $\mathfrak{t}^d_{\Z}\subset\mathfrak{t}^d$ be the integer lattices. Let $\{e_1,\ldots,e_n\}\subset\mathfrak{t}^n_{\Z}$ be an integer basis and let $\{\check{e}_1,\ldots,\check{e}_n\}\subset (\mathfrak{t}^n_{\Z})^*$ be the dual basis. Given a collection $u=\{u_1,\ldots,u_n\}\subset\mathfrak{t}^d_{\Z}$ of $n$ integer vectors that span $\mathfrak{t}^d_{\Z}$ over $\Z$, we define a map $\pi:\mathfrak{t}^n\to\mathfrak{t}^d$ by $\pi(e_i)=u_i$. We have the following exact sequences:
\begin{equation}
\label{ses1}
0\longrightarrow\mathfrak{k}\overset{\iota}{\longrightarrow}\mathfrak{t}^n\overset{\pi}{\longrightarrow}\mathfrak{t}^d\longrightarrow 0,
\end{equation}
\begin{equation}
\label{ses2}
0\longleftarrow(\mathfrak{k})^*\overset{\iota^*}{\longleftarrow}(\mathfrak{t}^n)^*\overset{\pi^*}{\longleftarrow}(\mathfrak{t}^d)^*\longleftarrow 0,
\end{equation}
where $\mathfrak{k}=\ker{\pi}$, and (\ref{ses2}) is the dual sequence of \ref{ses1}. Exponentiating (\ref{ses1}) gives an exact sequence of real tori
\begin{equation}
\label{ses3}
0\longrightarrow K\longrightarrow T^n\longrightarrow T^d\longrightarrow 0.
\end{equation}

Let $T^*\C^n$ be equipped its standard hyper-Kähler structure. Let $(z,w)=(z_1,w_1,\ldots,z_n,w_n)$ be the standard coordinates on $T^*\C^n$.  We consider $T^*\C^n$ equipped with the Kähler form $\omega_{\R}$ 
\[
\omega_{\R}=\frac{\sqrt{-1}}{2}\sum_{i=1}^n (dz_i\wedge d\bar{z}_i+dw_i\wedge d\bar{w}_i),
\]
and holomorphic symplectic form $\omega_{\C}$ 
\[
\omega_{\C}=\sum_{i=1}^n dz_i\wedge dw_i.
\]
Let $\vec{t}=(t_1,\ldots,t_n)\in T^n$ act on $T^*\C^n$ by 
\[
\vec{t}\cdot(z,w)=(t_1z_1,t_1^{-1}w_1,\ldots,t_nz_n,t_n^{-1}w_n),
\]
preserving the hyper-Kähler structure. The hyper-Kähler moment map 
\[
(\mu_{\R},\mu_{\C}): T^*\C^n \to (\mathfrak{k})^*\oplus (\mathfrak{k}_{\C})^*
\] 
for the restriction to $K$ of the $T^n$-action on $T^*\C^n$ is given by
	\[
	\mu_{\R}(z,w)=\frac{1}{2}\sum_{i=1}^n(|z_i|^2-|w_i|^2)\iota^*\check{e}_i, \qquad \mu_{\C}(z,w)=\sum_{i=1}^n(z_iw_i)\iota^*_{\C}\check{e}_i.
	\]
		\begin{definition}
		\label{def:hypertoric}
		Given a collection of primitive integer vectors $u$ and parameters $\lambda=(\lambda_{\R},\lambda_{\C})\in(\mathfrak{k})^*\oplus (\mathfrak{k}_{\C})^*$, the 
		hyper-Kähler quotient 
		\[
		\mathfrak{M}_{u,\lambda}=\left(\mu_{\R}^{-1}(\lambda_{\R})\cap\mu_{\C}^{-1}(\lambda_{\C})\right)/K
		\]
		is called a \textit{hypertoric variety}\footnote{A usual convention is setting $\lambda_{\C}=0$ in the definition.  In this paper we work with a generic complex structure and do not make this assumption.}. 
	\end{definition}
	
	Alternatively, $\mathfrak{M}_{u,\lambda}$ can be constructed as the GIT quotient
	\[
	\mathfrak{M}_{u,\lambda}=\mu_{\C}^{-1}(\lambda_{\C})\twobar_{\lambda_{\R}} K_{\C}=\mathrm{Proj} \left(\bigoplus_{k=0}^{\infty}\C[\mu_{\C}^{-1}(\lambda_{\C})]^{\lambda_{\R}^k}\right),
	\]
	where $K_{\C}$ the complexification of $K$, and $\lambda_{\R}\in(\mathfrak{k})^*$ is understood as a character $\lambda_{\R}:K_{\C}\to \C^{\times}$.
	
	The quotient torus $T^d=T^n/K$ acts on $\mathfrak{M}_{u,\lambda}$ with the hyper-Kähler moment map 
	\[
	(\bar{\mu}_{\R},\bar{\mu}_{\C}):\mathfrak{M}_{u,\lambda}\to (\mathfrak{t}^d)^*\oplus (\mathfrak{t}^d_{\C})^*
	\]
     given by 
	\[
	(\bar{\mu}_{\R},\bar{\mu}_{\C})[z,w]=\frac{1}{2}\sum_{i=1}^n(|z_i|^2-|w_i|^2+\hat{\lambda}_{\R,i})\check{e}_i\oplus\sum_{i=1}^n(z_iw_i+\hat{\lambda}_{\C,i})\check{e}_i\in\Kernel{(\iota^*)}\oplus\Kernel{(\iota^*_{\C})}=(\mathfrak{t}^d)^*\oplus (\mathfrak{t}^d_{\C})^*,
	\]
	where $((\hat{\lambda}_{\R,1},\ldots,\hat{\lambda}_{\R,n}),(\hat{\lambda}_{\C,1},\ldots,\hat{\lambda}_{\C,n}))\in(\mathfrak{t}^n)^*\oplus (\mathfrak{t}^n_{\C})^*$ is a lift of $\lambda$. Note that this map is always surjective.
	
\begin{example}
	Let $\{u_1,\ldots,u_d\}\subset\mathfrak{t}^d$ be a primitive integer basis. Define the map $\pi:\mathfrak{t}^{d+1}\to\mathfrak{t}^d$ by $\pi(e_i)=u_i$ for $i=1,\ldots,d$, and $\pi(e_{d+1})=u_{d+1}:=\sum_{j=1}^d (-u_j)$. $K\hookrightarrow T^{d+1}$ is then the diagonal sub-torus. If we set $\lambda_{\R}\in(\mathfrak{k})^*$ to be a regular value, and $\lambda_{\C}=0$, then the hypertoric variety $\mathfrak{M}_{u,\lambda}$ is $T^*\mathbb{P}^d$ (equipped with the standard complex structure).
\end{example}
	
	\begin{example}
		Let $u_1\in\mathfrak{t}^1$ be a primitive integer vector. Define $\pi:\mathfrak{t}^{n+1}\to\mathfrak{t}^1$ by $\pi(e_i)=u_1$ for $i=1,\ldots,n+1$. $K\hookrightarrow T^{n+1}$ is then the subtorus
		\[
		K=\{(t_1,\ldots,t_{n+1})\in T^{n+1}| \prod_{i=1}^{n+1} t_i=1\}.
		\]
		For $\lambda_{\R}$ a regular value and $\lambda_{\C}=0$, the hypertoric variety $\mathfrak{M}_{u,\lambda}$ is $\widetilde{\C^2/\Z_{n+1}}$, the crepant resolution of $A_{n}$ singularity $\C^2/\Z_{n+1}$.	
	\end{example}
	
	\subsection{Hyperplane arrangements}
	Let $\mathfrak{M}_{u,\lambda}$ be a hypertoric variety. Denote by $\mathcal{H}_{\R}=\{H_{\R,i}\}_{i=1}^n$ and $\mathcal{H}_{\C}=\{H_{\C,i}\}_{i=1}^n$ the collections of hyperplanes
	\[
	H_{\R,i}=\{s\in(\mathfrak{t}^d)^*|\left<s,u_i\right>-\hat{\lambda}_{\R,i}=0\},
	\]
	and
	\[
	H_{\C,i}=\{v\in(\mathfrak{t}^d_{\C})^*|\left<v,u_i\right>-\hat{\lambda}_{\C,i}=0\}.
	\]
 $\mathcal{H}_{\R}$ and $\mathcal{H}_{\C}$ are called the \textit{associated hyperplane arrangements} of $\mathfrak{M}_{u,\lambda}$. The hyperplane arrangements $\mathcal{H}_{\R}$ and $\mathcal{H}_{\C}$ are independent of the choice of the lift of $\lambda$ up to a translation and determine $\mathfrak{M}_{u,\lambda}$ up to a canonical isomorphism.

The following definition is important for smoothness of hypertoric varieties.

\begin{definition}
A hyperplane arrangement $\mathcal{H}_{\R}$(resp. $\mathcal{H}_{\C}$) is called \textit{simple} if every subset of $k$ hyperplanes with nonempty intersection intersects in codimension $k$. $\mathcal{H}_{\R}$(resp. $\mathcal{H}_{\C}$) is called \textit{unimodular} if every collection of d linearly independent vectors $\{u_{i_1},\ldots,u_{i_d}\}$ spans $\mathfrak{t}^d_{\Z}$ over $\Z$. 
\end{definition}

\begin{remark}
\label{rmk:holfib}
The holomorphic moment map
$\bar{\mu}_{\C}:\mathfrak{M}_{u,\lambda}\to(\mathfrak{t}^d_{\C})^*$ is a holomorphic $(\C^{\times})^d$-fibration, i.e. generic fibers of $\bar{\mu}_{\C}$ are biholomorphic to $(\C^{\times})^d$. If $v_0\in (\mathfrak{t}^d_{\C})^*$ is a point such that $v_0\in\bigcap_{i\in I}H_{\C,i}$ for some nonempty subset $I\subset\{1,\ldots,n\}$ and $v_0\notin H_{\C,i}$ for $i\notin I$, then $\bar{\mu}_{\C}^{-1}(v_0)\cong (\C\cup_0\C)^{\min\{|I|,d\}}\times(\C^{\times})^{\max\{d-|I|,0\}}$, where $\C\cup_0\C$ denotes the union of two affine lines intersecting transversely at the origin.

\end{remark}

See Figure \ref{fig:hyperplane-arrangement} for examples of hyperplane arrangements.

\begin{figure}[htb!]
	\includegraphics[scale=0.5]{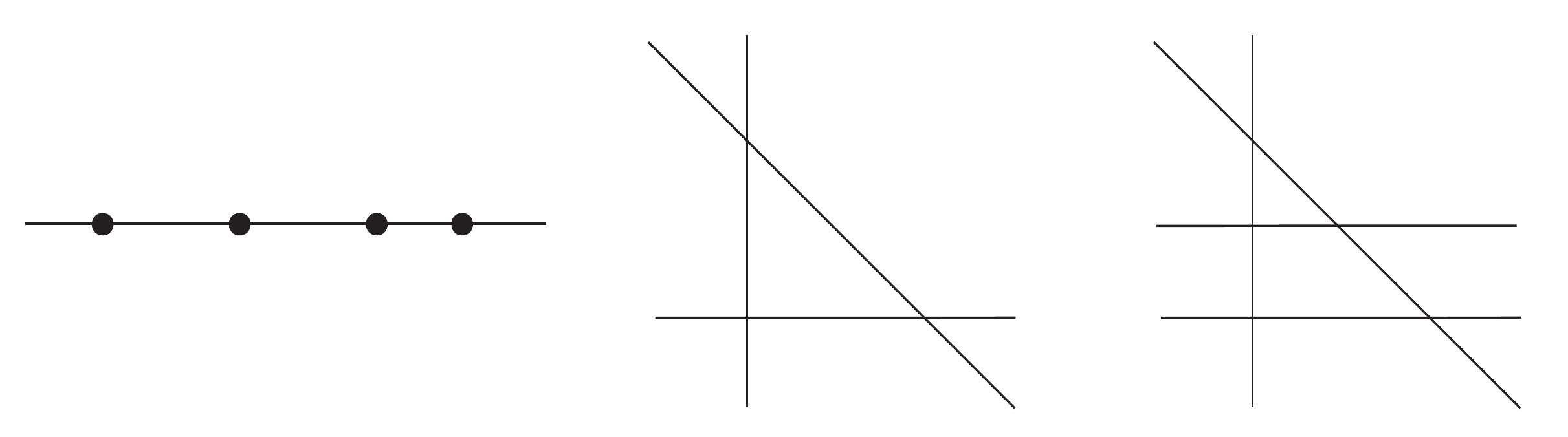}
	\caption{Examples of hyperplane arrangements.  The left corresponds to $\widetilde{\C^2/\Z_{4}}$ resolutions.  The middle corresponds to $T^*\C\bP^2$.  The right corresponds to a hypertoric variety which contains both $T^*\mathbb{F}_1$ and $T^*\C\bP^2$}.
	\label{fig:hyperplane-arrangement}
\end{figure}
	
	\subsection{Geometry and topology of hypertoric varieties}
Let $\mathcal{A}=\{A_i\}_{i=1}^n$ be the collection of affine subspaces $A_i=H_{\R,i}\times H_{\mathbb{C},i}\subset (\mathfrak{t}^d)^*\oplus (\mathfrak{t}^d_{\C})^*$. We have the following necessary and sufficient conditions for $\mathfrak{M}_{u,\lambda}$ to be an orbifold or smooth manifold:
	
	\begin{theorem}[{\cite[Theorem 3.2, 3.3]{BD}}]
		\label{thm:smooth}
		$\mathfrak{M}_{u,\lambda}$ is an orbifold with at worst Abelian quotient singularities if and only if every $d+1$ distinct elements in $\mathcal{A}$ have empty intersection. It is a smooth manifold if and only if, in addition, whenever $d$ distinct elements $A_{i_1},...,A_{i_d}$ have nonempty intersection, the set $\{u_{i_1},\ldots,u_{i_d}\}$ spans $\mathfrak{t}^d_{\Z}$ over $\Z$. 
	\end{theorem}
	
	\begin{corollary}
		For $\lambda_{\C}=0$, $\mathfrak{M}_{u,\lambda}$ is a smooth manifold if and only if $\mathcal{H}_{\R}$ is both simple and unimodular.
	\end{corollary}

	\begin{remark}
		The expression for the SYZ mirror in Theorem \ref{thm:SYZmir} still makes sense even when the hyperplane arrangements are not simple nor unimodular.  We speculate that it is useful for the study of hypertoric degenerations.
	\end{remark}  

	  	For a generic choice of $\lambda_{\C}$, $\mathfrak{M}_{u,\lambda}$ is simply an affine variety.
	  	
	\begin{theorem}[{\cite[Theorem 5.1]{BD}}]
		\label{thm:cplxstr}
		Let $\mathfrak{M}_{u,\lambda}$ be a hypertoric orbifold, and suppose $\mathcal{H}_{\C}$ is simple. Then, $\mathfrak{M}_{u,\lambda}$ equipped with the complex structure inherited from $T^*\C^n$ is biholomorphic to affine variety $\mathrm{Spec}\left(\C[W]^{K_{\C}}\right)$, where $W\subset T^*\C^n\times\C^d$ is defined by the equations
		\[
		z_iw_i=\left<v,u_i\right>-\hat{\lambda}_{\C,i}, \quad i=1,\ldots,n,
		\]
		and $K_{\C}$ acts on $T^*\C^n\times\C^d$ by $\vec{t}\cdot(z,w,v)=(t_1 z_1,t_1^{-1}w_1,\ldots,t_nz_n,t_n^{-1}w_n,v)$.
	\end{theorem}
	
		In general it is difficult to write down an explicit hyper-Kähler metric.  For a hypertoric variety, the K\"ahler metric is descended from the standard metric on $T^*\C^n$ and has a simple expression.
	
	\begin{theorem}[{\cite[Theorem 8.3]{BD}}]
		Let $s_i=|z_i|^2-|w_i|^2$, $v_i=z_iw_i$, and $r_i=\sqrt{s_i^2+4v_i\bar{v}_i}$. Then, on the open dense subset of $\mathfrak{M}_{u,\lambda}$ where the $T^d$-action is free, the induced Kähler form $\omega$ is given by
		\begin{equation}
		\label{KP}
		\omega=\frac{1}{4}dd^c(2\bar{\mu}_{\R},\bar{\mu}_{\C})^*\left(\sum_{i=1}^n(r_i+2\hat{\lambda}_{\R,i}\ln(s_i+r_i))\right),
		\end{equation}
		where $d^c=\sqrt{-1}(\bar{\partial}-\partial)$.
	\end{theorem}
	
	\subsection{Circuits and primitive curve classes}
	\label{sec:circuits}
	
		The SYZ mirrors that we are going to construct depend on K\"ahler parameters, which are recording the symplectic areas of primitive curve classes in $H_2(\mathfrak{M}_{u,\lambda};\Z)$. The following definition is crucial to understand primitive curve classes in hypertoric varieties.
	\begin{definition}[{\cite{MS}}]
	A \textit{circuit} $S\subset \{1,\ldots,n\}$ in $\mathcal{H}_{\R}$ is a minimal subset satisfying
	$$\bigcap_{i\in S} H_{\R,i}=\emptyset.$$
\end{definition}

A circuit $S$ admits a unique splitting $S=S^+\coprod S^-$ (up to swapping $S^+$ and $S^-$), which is characterized by the equation	
\[
\sum_{i\in S^+} u_i-\sum_{i\in S^-	} u_i=0 \in \mathfrak{t}^d.
\]
For each circuit $S$, we fix the splitting $S=S^+\coprod S^-$ such that if we set
\[
\beta_{S}=\sum_{i\in S^+} e_i - \sum_{i\in S^-} e_i,
\]
then $\hat{\lambda}_{\R}(\beta_{S})>0$
for any lift $\hat{\lambda}_{\R}\in (\mathfrak{t}^n)^*$ of $\lambda_{\R}$.  

$\beta_{S}$ is a primitive class in	$\mathfrak{k}_{\Z}=H_2(\mathfrak{M}_{u,\lambda};\Z)$.  It can be understood as a curve class obtained from gluing holomorphic discs emanated from the hyperplanes indexed by $S$.  We denote by $q^{\beta_{S}}$ the Kähler parameter associated to $\beta_{S}$.
	
	\subsection{Cotangent bundles of toric varieties in a hypertoric variety}
    Let $\mathcal{H}_{\R}$ be the real hyperplane arrangement of $\mathfrak{M}_{u,\lambda}$. Let $\Delta$ be a convex polytope in $(\mathfrak{t}^d)^*$ with its interior being a chamber in the complement of $\mathcal{H}_{\R}$. We will assume $\Delta$ is simple, which is the case when $\mathcal{H}_{\R}$ is simple. We further assume $\lambda_{\C}=0$, so that $\mathfrak{M}_{u,\lambda}$ is equipped with its canonical complex structure. Then, the cotangent bundle $T^*X_{\Delta}$ of the toric variety $X_{\Delta}$ naturally embeds into $\mathfrak{M}_{u,\lambda}$ as an open dense subset (Fig. \ref{fig:cotang-in-hypertoric}).
	
	\begin{theorem}[{\cite[Theorem 7.1]{BD}}]
		\label{thm:cotangent}
		$T^*X_{\Delta}$ with its canonical holomorphic-symplectic structure is $T^d$-equivariantly isomorphic to an open dense subset $U_{\Delta}$ of $\mathfrak{M}_{u,\lambda}$. The hyper-Kähler metric of $\mathfrak{M}_{u,\lambda}$ restricted to the zero section of $T^*X_{\Delta}$ is the Kähler metric on $X_{\Delta}$ determined by $\Delta$.
	\end{theorem}
	 $T^*X_{\Delta}\subset \mathfrak{M}_{u,\lambda}$ was constructed in \cite{BD} as follows. For simplicity, let's fix a lift $$((\hat{\lambda}_{\R,1},\ldots,\hat{\lambda}_{\R,n}),(\hat{\lambda}_{\C,1},\ldots,\hat{\lambda}_{\C,n}))\in(\mathfrak{t}^n)^*\oplus (\mathfrak{t}^n_{\C})^*$$
	of $\lambda$ such that $\hat{\lambda}_{\R,i}=0$ and $\hat{\lambda}_{\C,i}=0$ for $i=1,\ldots,d$. Let $\{H_{\R,i}^+\}_{i=1}^n$ and $\{H_{\R,i}^-\}_{i=1}^n$ be the half-spaces
	\[
	H_{\R,i}^+=\{s\in (\mathfrak{t}^d)^*|\left<s,u_i\right>-\hat{\lambda}_{\R,i}\ge 0\},
	\]
	\[
	H_{\R,i}^-=\{s\in (\mathfrak{t}^d)^*|\left<s,u_i\right>-\hat{\lambda}_{\R,i}\le 0\}.
	\]
	Let $\sigma:\{1,...,n\}\to \{+,-\}$ be the sign vector such that
	\[
	\Delta=\bigcap_{i=1}^n H_{\R,i}^{\sigma(i)},
	\]
	and let $\bar{\sigma}$ be the sign vector such that $\bar{\sigma}(i)\ne \sigma(i)$ for all $i$. Each face $F$ of $\Delta$ is given by an intersection of hyperplanes
	$\bigcap_{i\in I} H_{\R,i}$, for some $I\subset\{1,\ldots,n\}$. For $F$ a face of $\Delta$, we define a subset $Y_F\subset T^*\C^n$ by
	\[
	Y_F=\{(z,w)\in T^*\C^n|z_i=0\iff i\in I \text{ and }\sigma(i)=+; w_i=0\iff i\in I \text{ and }\sigma(i)=-\}.
	\]
	In particular, if $F$ is the codimension-$0$ face, we have $I=\emptyset$, and 
	\[
Y_{F}=\{(z,w)\in T^*\C^n|z_i\neq 0 \text{ if }\sigma(i)=+; w_i\neq 0, \text{ if }\sigma(i)=-\}.
	\]
	We define a $T^n$-invariant subset $Y_{\Delta}\subset T^*\C^n$ to be the union
	\[
	Y_{\Delta}=\bigcup_F Y_F,
	\]
	where the union is over all faces $F$ of $\Delta$. $T^*X_{\Delta}\subset\mathfrak{M}_{u,\lambda}$ is then constructed by restricting the hyper-Kähler quotient construction to $Y_{\Delta}$,
	\[
	T^*X_{\Delta}=\left(Y_{\Delta}\cap\mu_{\R}^{-1}(\lambda_{\R})\cap\mu_{\C}^{-1}(0)\right)/K.
	\]
	
	We provide here an explicit description of the complement of $T^*X_{\Delta}$ in $\mathfrak{M}_{u,\lambda}$ in term of its moment map image. Let $\mathfrak{J}$ be the collection of all subsets $J\subset\{1,\ldots,n\}$ such that the intersection $\bigcap_{j\in J} H_j$ is nonempty, and is not a face of $\Delta$. Denote by $\Delta_J$ the polytope 
	\[
	\Delta_J=\bigcap_{j\in J}H^{\bar{\sigma}(j)}_{\R,j}.
	\]
	Notice that $\Delta_J$ is non-adjacent to $\Delta$.
	\begin{prop}
		\label{prop:complement}
		The complement of $T^*X_{\Delta}$ in $\mathfrak{M}_{u,\lambda}$ is the union $V_{\Delta}=\bigcup_{J\in \mathfrak{J}} V_J$, where
		\[
		V_J=(\bar{\mu}_{\R},\bar{\mu}_{\C})^{-1}\left(\Delta_J\times \bigcap_{j\in J}H_{\C,j}\right).
		\]	
	\end{prop}
	
	\begin{proof}
		Let $J\in\mathfrak{J}$, and denote by $Y_J\subset T^*\C^n\setminus Y_{\Delta}$ the subset
		\[
		Z_J=\{(z,w)\in T^*\C^n|z_j=0\iff j\in J \text{ and }\sigma(j)=+; w_j=0\iff j\in J \text{ and }\sigma(j)=-\}.
		\]
		Restricting the hyper-Kähler quotient construction to $Z_J$ gives 
		\[
		V_J=\left(Y_J\cap\mu_{\R}^{-1}(\lambda_{\R})\cap\mu_{\C}^{-1}(0)\right)/K\subset\mathfrak{M}_{u,\lambda}\setminus T^*X_{\Delta}.
		\]
		By construction, we have $\mathfrak{M}_{u,\lambda}\setminus T^*X_{\Delta}=\bigcup_{J\in \mathfrak{J}} V_J$. The image of $V_J$ under the hyper-Kähler moment map $(\bar{\mu}_{\R},\bar{\mu}_{\C})$ is $\Delta_J\times \bigcap_{j\in J} H_{\mathbb{C},j}$. To see that it is disjoint from the image of $T^*X_{\Delta}$, suppose we have $[z,w]\in T^*X_{\Delta}$ with $\bar{\mu}_{\C}([z,w])\in\bigcap_{j\in J} H_{\C,j}$, since $J$ does not define a face of $\Delta$, we must have $z_j\ne 0, w_j= 0$ and $\sigma(j)=+$ or $z_j=0, w_j\ne 0$ and $\sigma(j)=-$ for some $j\in J$, but then $\bar{\mu}_{\R}([z,w])\notin H^{\bar{\sigma}(j)}_{\R,j}\supset\Delta_J$.
	\end{proof}

	\begin{figure}[htb!]
		\includegraphics[scale=0.5]{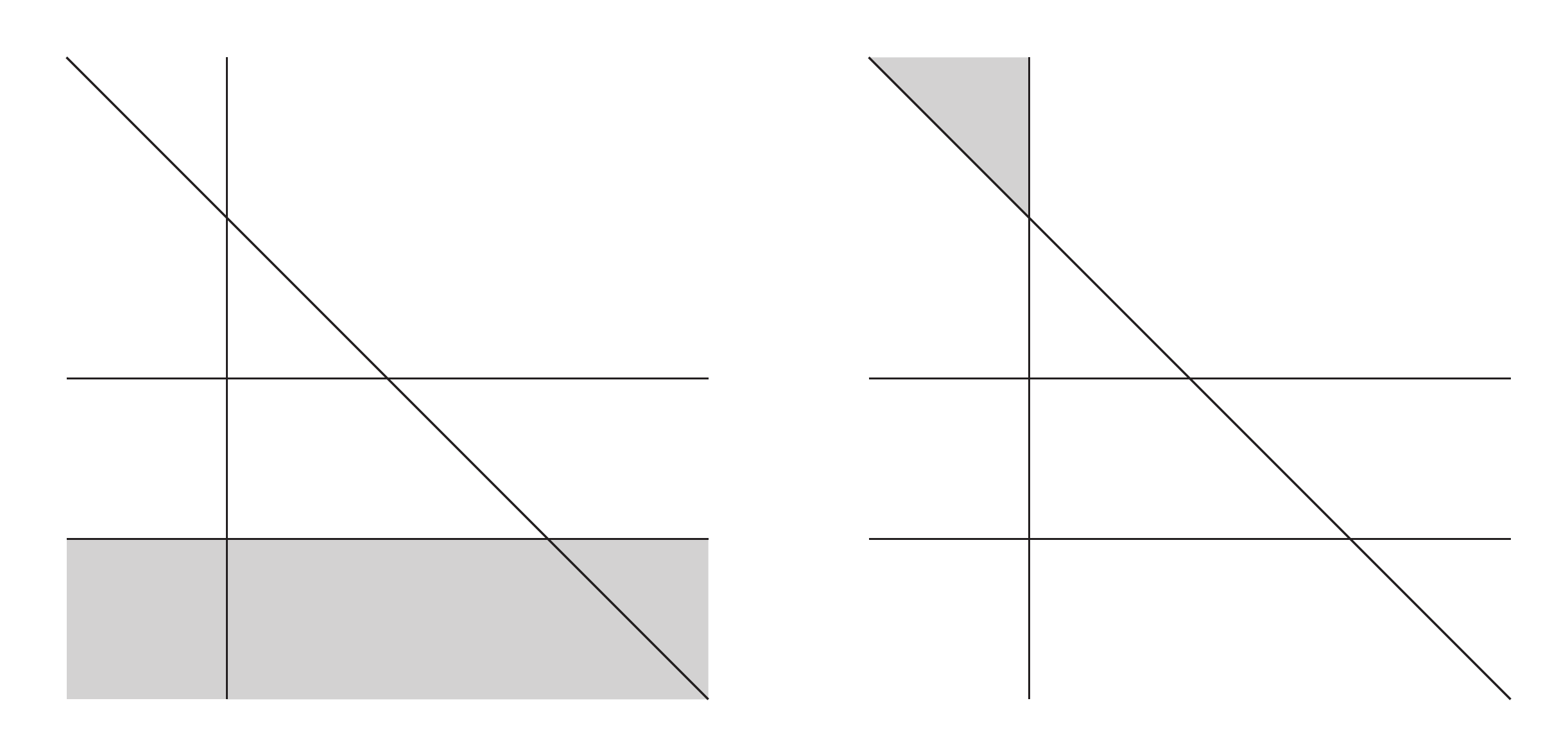}
		\caption{A hypertoric manifold that contains both $T^*\C\bP^2$ and $T^*\mathbb{F}_1$. The closure of the shaded region on the left (resp. right) corresponds to the image of the complement of $T^*\C\bP^2$ (resp. $T^*\mathbb{F}_1$) under $\bar{\mu}_{\R}$.}
		\label{fig:cotang-in-hypertoric}
	\end{figure}
	
		In this paper we work with smooth hypertoric varieties. In addition to $\mathfrak{M}_{u,\lambda}$ being smooth, we shall assume $\mathcal{H}_{\C}$ to be simple for the rest of this paper. When $\mathcal{H}_{\C}$ is simple, under the unimodularity assumption, $\mathfrak{M}_{u,\lambda}$ is smooth for all choices of $\lambda_{\R}$ by Theorem \ref{thm:smooth}. We do not assume $\mathcal{H}_{\R}$ to be simple.
		
	\section{Lagrangian torus fibrations on hypertoric varieties} \label{sec:fib}
	In this section, we construct piecewise smooth Lagrangian torus fibrations on hypertoric varieties. It was first suggested by Joyce in \cite{Joyce-sing} that special Lagrangian fibrations should in general be piecewise smooth. In \cite{AAK}, Abouzaid, Auroux and Katzarkov constructed piecewise smooth Lagrangian torus fibrations on the anticanonical divisor complement $X^0$ of the blowup $X=\mathrm{Bl}_{H\times\{0\}}V\times\C$, where $V$ is a toric variety and $H\subset V$ is a hypersurface, by pulling back Lagrangian torus fibrations on the symplectic reductions of $X^0$ (which are isomorphic to the open dense torus orbit $V^0\subset V$) and assembling them together. This construction is similar to those previously considered by Gross \cite{Gross-eg}, Goldstein \cite{Goldstein}, Castaño-Bernard and Matessi \cite{CBM1,CBM2}. The additional technical input in \cite{AAK} was the use of Moser's trick to interpolate between the reduced (possibly singular) K\"ahler forms and the torus-invariant K\"ahler forms on $V^0$.
	
	\subsection{Lagrangian torus fibrations on the reduced spaces}
	Denote by $s=(s_1,\ldots,s_n)$ the standard coordinates on $(\mathfrak{t}^n)^*$ rescaled by a factor of $2$, and $v=(v_1,\ldots,v_n)$ the standard complex-coordinates on $(\mathfrak{t}^n_{\C})^*$. We first construct Lagrangian torus fibrations on the symplectic reductions $X_s$ of $\mathfrak{M}_{u,\lambda}$ at level $\frac{s}{2}\in (\mathfrak{t}^d)^*\subset (\mathfrak{t}^n)^*$. $X_{s}$ can be constructed as
	\[
	X_{s}=\bar{\mu}_{\R}^{-1}\left(\frac{s}{2}\right)/(T^n/K).
	\]
	                                                    
	For simplicity, we will assume from now on that the vectors $u_1,\ldots,u_d$ are linearly independent, and write $u_{\ell}=\sum_{i=1}^d a_{\ell i}u_i$ for $\ell=d+1,\ldots,n$. The coefficients $a_{\ell i}$ are integers, since $\{u_1,\ldots,u_d\}$ spans $\mathfrak{t}^d_{\Z}$ over $\Z$. We also fix a lift $((\hat{\lambda}_{\R,1},\ldots,\hat{\lambda}_{\R,n}),(\hat{\lambda}_{\C,1},\ldots,\hat{\lambda}_{\C,n}))\in(\mathfrak{t}^n)^*\oplus (\mathfrak{t}^n_{\C})^*$ of $\lambda$ such that $\hat{\lambda}_{\R,i}=0$ and $\hat{\lambda}_{\C,i}=0$ for $i=1,\ldots,d$. We can then identify the $\bar{\mu}_{\C}:\mathfrak{M}_{u,\lambda}\to (\mathfrak{t}^d_{\C})^*$ with the map $(z_1w_1,\ldots,z_dw_d):\mathfrak{M}_{u,\lambda}\to \C^d$ via the projection to the first $d$ components. The restriction of $\bar{\mu}_{\C}$ to $\bar{\mu}_{\R}^{-1}(\frac{s}{2})$ descends to a biholomorphism $
      X_{s}\to (\mathfrak{t}^d_{\C})^*.$
	We can therefore identify the reduced spaces $X_{s}$ with $\C^d*$ equipped with complex-coordinates $(v_1,\ldots,v_d)$. We will abuse notations and write
	\[
	s_i=|z_i|^2-|w_i|^2, \quad v_i=z_iw_i,
	\]
	and set 
	\[
	r_i=\sqrt{s_i^2+4v_i\bar{v}_i}.
	\]
	for $i=1,\ldots,n$. These can be viewed as functions on $X_{s}$. In particular, $s_i$ are constants. The K\"ahler potential of the reduced K\"ahler form on $X_{s}$ has a simple expression in term of $s_i$ and $r_i$.
	
	\begin{lemma}
		The K\"ahler potentials $K_{red,s}$ for the reduced K\"ahler forms $\omega_{red,s}$ on $X_{s}$ are given by
		\begin{equation}
		\label{kp1}
		K_{red,s}=\frac{1}{4}\sum_{i=1}^n\left(r_i-s_i\ln|s_i\pm r_i|\right). \quad + \text{ if } s_i\ge 0, - \text{ otherwise}.
		\end{equation}
	\end{lemma}
	
	\begin{proof}
		Consider the action of $T^n$ and its complexification $(\C^{\times})^n$ restricted to the invariant subvariety $W=\mu_{\C}^{-1}(\lambda_{\C})\subset T^*\C^n$. $X_{s}$ can be obtained either as a symplectic reduction or a GIT quotient of $W$,
		\[
		X_{s}=(\tilde{\mu}_{\R}|_{W})^{-1}\left(\frac{s}{2}\right)/T^n=W\twobar_{\frac{s}{2}} (\C^{\times})^n,
		\]
		where $\tilde{\mu}_{\R}$ is the moment map for the $T^n$-action on $T^*\C^n$ with respect $\omega_{\R}$. For any $(z,w)\in W$, there exist a unique element $\vec{t}_{(z,w)}\in \exp(i\mathfrak{t}^n)$ such that $\vec{t}_{(z,w)}\cdot (z,w)\in (\tilde{\mu}_{\R}|_{W})^{-1}\left(\frac{s}{2}\right)$. Denote by $q:W\to (\tilde{\mu}_{\R}|_{W})^{-1}\left(\frac{s}{2}\right)$ the map $q(z,w)=\vec{t}_{(z,w)}\cdot (z,w)$, and by $p:(\tilde{\mu}_{\R}|_{W})^{-1}\left(\frac{s}{2}\right)\to X_{s}$ the quotient map. Let $\hat{\omega}$ be the pull-back of $\omega_{red,s}$ on $W$ via $p\circ q$. Let $\chi_{\frac{s}{2}}:(\C^{\times})^n\to\C^{\times}$ the character given by $\frac{s}{2}$. By \cite[Theorem 7]{BG}, we have $\hat{\omega}=dd^c\hat{K}$, for a $T^n$-invariant function $\hat{K}$ on $W$ defined as
		\begin{equation}
		\label{kp2}
		\hat{K}(z,w)=K_0(\vec{t}_{(z,w)}\cdot (z,w))+\frac{1}{4\pi}\ln|\chi_{\frac{s}{2}}(\vec{t}_{(z,w)})|^2,
		\end{equation}
		where $K_0$ is the K\"ahler potential $\frac{1}{4}\sum_{i=1}^n |z_i|^2+|w_i|^2$ restricted to $W$. We have
        \begin{equation}
		\label{eq1}
		K_0\left(\vec{t}_{(z,w)}\cdot (z,w)\right)=\frac{1}{4}\sum_{i=1}^n \sqrt{s_i^2+4v_i\bar{v}_i},
       \end{equation}
		whereas
		\[
		|\chi_{\frac{s}{2}}(t_{(z,w)})|^2=\prod_{i=1}^n |\vec{t}_{(z,w),i}|^{-2\pi s_i}.
		\]
		$\vec{t}_{(z,w),i}$ is determined by
		\[
		\left|(\vec{t}_{(z,w),i}z_i\right|^2-\left|\vec{t}_{(z,w),i}^{-1}w_i\right|^2=s_i.
		\]
		This means
		\[
		|\vec{t}_{(z,w),i}|^2=\cfrac{s_i\pm \sqrt{s_i^2+4|z_i|^2|w_i|^2}}{2|z_i|^2}.
		\]
		Thus,
        \begin{equation}
		\label{eq2}
		\frac{1}{4\pi}\ln|\chi_{\frac{s}{2}}(\vec{t}_{(z,w)})|^2=\frac{1}{4}\sum_{i=1}^n -s_i\ln\left|s_i\pm \sqrt{s_i^2+4|z_i|^2|w_i|^2}\right|+s_i\ln(2|z_i|^2).
       \end{equation}
		Notice that $s_i$ in (\ref{eq1}) and (\ref{eq2}) are constants. Denote by $\iota:(\tilde{\mu}_{\R}|_{W})^{-1}\left(\frac{s}{2}\right)\to W$ the inclusion map. Since $dd^c\ln(2|z_i|^2)=0$, the terms $\frac{1}{4}s_i\ln(2|z_i|^2)$ do not contribute to $\hat{\omega}$. Thus, we have
		\[
		p^*\omega_{red,s}=\iota^*\hat{\omega}=\iota^*dd^c\left(\hat{K}(z,w)-\frac{1}{4}\sum_{i=1}^n s_i\ln(2|z_i|^2)\right)=\iota^*dd^c(p\circ q)^*K_{red,s}=p^*dd^c K_{red,s}.
		\]
	\end{proof}
	
	\begin{remark}
		If we view
		\[
		F=\frac{1}{4}\sum_{i=1}^n \left(r_i-s_i\ln|s_i+r_i|\right)
		\]
		as a function on $(\mathfrak{t}^n)^*\oplus (\mathfrak{t}^n_{\C})^*$, it is then the Legendre transform of the Kähler potential $\frac{1}{4}\sum_{i=1}^n|z_i|^2+|w_i|^2$ on $T^*\C^n$.  In \cite{BD}, (\ref{KP}) was obtained as the Legendre transform of $F$ restricted to the subspace $(\mathfrak{t}^d)^*\oplus (\mathfrak{t}^d_{\C})^*\subset (\mathfrak{t}^n)^*\oplus (\mathfrak{t}^n_{\C})^*$. We can alternatively derive (\ref{kp1}) as the Legendre transform of $F$ further restricted to the subspace $\{\frac{s}{2}\}\times(\mathfrak{t}^d_{\C})^*\subset (\mathfrak{t}^n)^*\oplus (\mathfrak{t}^n_{\C})^*$. 
	\end{remark}
	
	The reduced K\"ahler forms $\omega_{red,s}$ are singular along the hyperplane $H_{\C,i}$, when $s\in H_{\R,i}$. They are also not invariant under any obvious $T^d$-action on $X_{s}\cong\C^d$. These obstacles to constructing Lagrangian torus fibrations on the reduced spaces were also encountered in \cite{AAK}. We will use their strategy to construct Lagrangian torus fibrations on $X_{s}$.
	
	We first introduce an explicit family of smoothing $\omega_{sm,s}$ of $\omega_{red,s}$:
	\begin{equation}
	\label{kp3}
	\omega_{sm,s}=dd^c K_{sm,s}:=\frac{1}{4}dd^c \left(\sum_{i=1}^n\sqrt{s_i^2+4v_i\bar{v}_i+\kappa^4}-s_i\ln\left|s_i+\sqrt{s_i^2+4v_i\bar{v}_i+\kappa^4}\right|\right),
	\end{equation}
	where $\kappa>0$ is an arbitrarily small constant. $\omega_{sm,s}$ is Kähler by construction. Since $H^2(X_s;\R)=0$, we have $[\omega_{sm,s}]=[\omega_{red,s}]$. We write $v_{\ell}=\sum_{i}^d a_{\ell i}v_i+b_{\ell}$ for $\ell=d+1,\ldots,n$, where $b_{\ell}\in\C$ are constants determined by $\lambda_{\C}$. Notice that the terms
	\[
	\frac{1}{4}dd^c \left(\sum_{\ell=d+1}^n\sqrt{s_{\ell}^2+4v_{\ell}\bar{v}_{\ell}+\kappa^4}-s_i\ln\left|s_{\ell}+\sqrt{s_{\ell}^2+4v_{\ell}\bar{v}_{\ell}+\kappa^4}\right|\right)
	\]
	in (\ref{kp3}) are not invariant under the standard $T^d$-action centered at a point in $\C^d$. To remedy this, let $c=(c_1,\ldots,c_d)\in\C^d$ be a point away from the hyperplanes in $\mathcal{H}_{\C}$, and $T^d$ acts on $\C^d$ by the standard action centered at $c$. We isotope $\omega_{sm,s}$ to the family of $T^d$-invariant Kähler form $\omega_{inv,s}$ defined by averaging $\omega_{sm,s}$ over the $T^d$-action,
	\[
	\omega_{inv,s}=\cfrac{1}{(2\pi)^d}\int_{g\in T^d} g^*\omega_{sm,s}dg.
	\]
	Since $\omega_{inv,s}$ is the exterior derivative of a $T^d$-invariant $1$-form (which is the $T^d$-average of $d^c K_{sm,s}$), its pullback to each $T^d$-orbit must vanish.  This means the $T^d$-orbits in $X_{s}$ are Lagrangian with respect to $\omega_{inv,s}$.
	
	We now prove the following lemma.
	
	\begin{lemma}
		\label{lemma:moser}
		There exists a family of homeomorphisms $(\phi_{s})_{s\in (\mathfrak{t}^d)^*}$ of $X_{s}$ such that
		\begin{enumerate}[label=\textnormal{(\arabic*)}]
			\item $\phi_s$ is a diffeomorphism if $s\notin H_{\R,i}$ for $i=1,\ldots,d$. It is a diffeomorphism away from $H_{\C,i}$ if $s\in H_{\R,i}$;
			\item $\phi_s$ intertwines the reduced (possibly singular)Kähler form $\omega_{red,s}$ and the $T^d$-invariant Kähler form $\omega_{inv,s}$;
			\item $\phi_{s}$ depends on $s$ continuously, and smoothly away from $\bigcup_{i=1}^n H_{\R,i}$.
		\end{enumerate}
	\end{lemma}
	
	\begin{proof}
		We construct $\phi_{s}$ as the composition of $\phi_{sm,s}$ and $\phi_{inv,s}$ such that $\phi_{sm,s}$ takes $\omega_{red,s}$ to $\omega_{sm,s}$, and $\phi_{inv,s}$ takes $\omega_{sm,s}$ to $\omega_{inv,s}$, each satisfying the desired properties.
		
		\textbf{Step 1.} We interpolate between $\omega_{red,s}$ and $\omega_{sm,s}$ via the family of Kähler forms $\omega_{t,s}$, $t\in[0,\kappa]$, defined by
		\begin{equation}
		\label{moser1}
		\omega_{t,s}=dd^c K_{t,s}:=\frac{1}{4}dd^c\left(\sum_{i=1}^n r_{t,i}-s_i\ln|s_i+r_{t,i}|\right),
		\end{equation}
		where $r_{t,i}=\sqrt{s_i^2+4v_i\bar{v_i}+t^4}$. We use Moser's trick and look for the vector field $V_{t,s}$ satisfying
		\[
		\mathcal{L}_{V_{t,s}}\omega_{t,s}+\frac{d}{dt}\omega_{t,s}=	\mathcal{L}_{V_{t,s}}\omega_{t,s}+dd^c\left(\cfrac{dK_{t,s}}{dt}\right)=0	.
		\]
		By Cartan's formula, we have
		\[
		d\iota_{V_{t,s}}\omega_{t,s}=-dd^c\left(\cfrac{dK_{t,s}}{dt}\right),
		\]
		from which we deduce
		
		\begin{equation*}
		\iota_{V_{t,s}}\omega_{t,s}=a_{t,s}:=-d^c\left(\cfrac{dK_{t,s}}{dt}\right)=-\frac{1}{2}d^c\left(\sum_{i=1}^n \cfrac{t^3}{s_i+r_{t,i}}\right).
		\end{equation*}
		
		We write $u_{\ell}=\sum_{i=1}^d a_{\ell i}u_i$ for $\ell=1,\ldots,n$, where $a_{\ell i}=\delta_{\ell i}$ for $\ell=1,\ldots,d$. We denote $\bm{i}=\sqrt{-1}$ so that it is not confused with the index $i$. We have
		\begin{equation*}
\omega_{t,s}=\sum_{1\le i,j\le d}\omega_{t,s,ij}dv_i\wedge d\bar{v}_j:=\bm{i}\sum_{1\le i,j\le d}\Bigg(\sum_{\ell=1}^n a_{\ell i}a_{\ell j}\left(\cfrac{(s_{\ell}+r_{t,\ell})r_{t,\ell}-2|v_{\ell}|^2}{(s_{\ell}+r_{t,\ell})^2r_{t,\ell}}\right)\Bigg)dv_i\wedge d\bar{v}_j,
		\end{equation*}
		
and
		\begin{equation*}
a_{t,s}=\sum_{i=1}^d a_{t,s,i} d\bar{v}_i-\bar{a}_{t,s,i} dv_i
:=\bm{i} \sum_{i=1}^d\left(\sum_{\ell=1}^n
\cfrac{t^3 a_{\ell i}v_{\ell}}{(s_{\ell}+r_{t,\ell})^2r_{t,\ell}}\right)d\bar{v}_i-\left(\sum_{\ell=1}^n \cfrac{t^3 a_{\ell i}\bar{v}_{\ell}}{(s_{\ell}+r_{t,\ell})^2r_{t,\ell}}\right)dv_i.
		\end{equation*}
		Denote by $A=(A_{ij})$ the matrix with entries $A_{ij}=\omega_{t,s,ij}$, and let $A^{-1}=(A^{ji})$ be its inverse. The vector field $V_{t,s}$ is then given by
		\[
		V_{t,s}=\sum_{j=1}^d f_{t,s,j}\frac{\partial}{\partial v_j}+g_{t,s,j}\frac{\partial}{\partial \bar{v}_j}=\sum_{j=1}^d\left(\sum_{i=1}^d A^{ji}
		a_{t,s,i}\right)\frac{\partial}{\partial v_j}+\left(\sum_{i=1}^d A^{ji}
		\bar{a}_{t,s,i}\right)\frac{\partial}{\partial \bar{v}_j}.
		\]
		$V_{t,s}$ is smooth except when $t=0$ and $s\in H_{\R,i}$, in which case it is singular along $H_{\C,i}$. We will show that the flow of $V_{t,s}$ is well-defined and $V_{t,s}$ is complete.
		
		Let $I\subset\{1,\ldots,n\}$ be a multi-index such that $\bigcap_{k\in I} H_{\C,k}\ne\emptyset$. Let $s\in(\mathfrak{t}^d)^*$ be a point such that $s\in H_{\R,k}$ if and only if $k \in I$, and let $v_0\in\bigcap_{k\in I} H_{\C,k}$. To analyze the singularities of the functions $f_{t,s,j}$ (the analysis for $g_{t,s,j}$ is identical and hence omitted), we consider the following limits:
		\begin{equation*}
		\label{limit}
		\lim_{(t,v)\to (0,v_0)}f_{t,s,j}=\lim_{(t,v)\to (0,v_0)}\sum_{i=1}^d A^{ji}a_{t,s,i}=\lim_{(t,v)\to (0,v_0)}\sum_{i=1}^d\frac{C_{ji}}{\det A}a_{t,s,i},
		\end{equation*}
		where $C_{ji}$ is the $(j,i)$-cofactor of $A$. Since $\bigcap_{k\in J} H_{\C,k}\ne\emptyset$ and the hyperplane arrangement $\mathcal{H}_{\C}$ is simple, the vectors $\{u_k\}_{k\in I}$ are linearly independent. Thus we can assume $I\subset\{1,\ldots,d\}$ by rearranging the indices (notice that the coefficients in $u_{\ell}=\sum_{i=1}^d a_{\ell i}u_i$, $\ell=d+1,\ldots,n$, will change accordingly).  		
		Let $A_I$ be the matrix obtained from $A$ by removing the $k^{\mathrm{th}}$ row and column for $k\in I$. Denote by $A_{I,ij}$ the matrix obtained from $A_I$ by removing the $j^{\mathrm{th}}$ row and the $i^{\mathrm{th}}$. Note that $\det A_I\ne 0$ since $A_I$ is positive-definite, and $\det A_{I,ij}$ is non-singular. As $(t,v)\to (0,v_0)$, $\det A$ is dominated by the term
		\[
		\left(\prod_{k\in I}\cfrac{(s_k+r_{t,k})r_{t,k}-2|v_k|^2}{(s_k+r_{t,k})^2r_{t,k}}\right)\times\det A_I.
		\]
		while $C_{ji}$ is dominated by the term
		\[
		\left(\prod_{k\in I\setminus\{i,j\}}\cfrac{(s_k+r_{t,k})r_{t,k}-2|v_k|^2}{(s_k+r_{t,k})^2r_{t,k}}\right)\times\det A_{I,ij},
		\]
		As $(t,v)\to (0,v_0)$, $C_{ji}$ blows up of at most the same order as $\det A$, while $a_{t,s,i}$ vanishes. This shows that $V_{t,s}$ extends continuously to be zero along its singular loci.
		
		On the other hand, let $g_{t,s}$ be the Kähler metric determined by $\omega_{t,s}$. Since the  Kähler metric on $\mathfrak{M}_{u,\lambda}$ is complete, $g_{t,s}$ is a complete metric whenever it is non-singular. Denote by $\norm{\cdot}_{t,s}$ be the norm with respect to $g_{t,s}$. Since $V_{t,s}$ is  dual vector field of $a_{t,s}$, we have
		\[
		\norm{V_{t,s}}_{t,s}=\norm{a_{t,s}}_{t,s}\le 2\sum_{i=1}^d |a_{t,s,i}|\norm{dv_i}_{t,s}=O(|v|^{-\frac{3}{2}})		\]
		as $|v|\to\infty$. Thus, $V_{t,s}$ is uniformly bounded with respect to $g_{t_0,s}$, $t_0>0$.. We can therefore define $\phi_{sm,s}$ to be the time-$\kappa$ flow generated by $V_{t,s}$.
		
		\textbf{Step 2.} We interpolate between $\omega_{sm,s}$ and $\omega_{inv,s}$ via the family of Kähler forms $\omega'_{t,s}$, $t\in[0,1]$, defined by
		\begin{equation*}
		\omega'_{t,s}=t\omega_{inv,s}+(1-t)\omega_{sm,s}=dd^c	\left(tK_{inv,s}+(1-t)K_{sm,s}\right),
		\end{equation*}
		where $K_{sm,s}$ is defined as in (\ref{kp3}), and
		\[
		K_{inv,s}=\cfrac{1}{(2\pi)^d}\int_{g\in T^d} g^*K_{sm,s}dg.
		\]
		We again use Moser's trick and look for the vector field $V'_{t,s}$ satisfying
		\[
		\mathcal{L}_{V'_{t,s}}\omega'_{t,s}+\frac{d}{dt}\omega'_{t,s}=0	.
		\]
		By Cartan's formula, we have
		\[
		d\iota_{V'_{t,s}}\omega'_{t,s}=\omega_{sm,s}-\omega_{inv,s},
		\]
		from which we deduce
		\begin{equation*}
		\label{mt1}
		\iota_{V'_{t,s}}\omega'_{t,s}=a'_{t,s}=d^c\left(K_{sm,s}-K_{inv,s}\right).
		\end{equation*}
		Writing out the relevant terms explicitly, we have
		\begin{multline*}
		\omega'_{t,s}=\sum_{1\le i,j\le d}\omega'_{t,s,ij}dv_i\wedge d\bar{v}_j:=\bm{i}\sum_{1\le i,j\le d}\Bigg(\sum_{\ell=1}^n \cfrac{t a_{\ell i}a_{\ell j}}{(2\pi)^d}\int_{g\in T^d} g^*\left(\cfrac{\left((s_{\ell}+r_{\kappa,\ell})r_{\kappa,\ell}-2|v_{\ell}|^2\right)}{(s_{\ell}+r_{\kappa,\ell})^2r_{\kappa,\ell}}\right)dg \\
		+ (1-t)a_{\ell i}a_{\ell j}\left(\cfrac{\left((s_{\ell}+r_{\kappa,\ell})r_{\kappa,\ell}-2|v_{\ell}|^2\right)}{(s_{\ell}+r_{\kappa,\ell})^2r_{\kappa,\ell}}\right)\Bigg)dv_i\wedge d\bar{v}_j,
		\end{multline*}	
		and
		\begin{multline*}
		a'_{t,s}=\sum_{i=1}^d a'_{t,s,i} d\bar{v}_i-\bar{a}'_{t,s,i} dv_i
		:=\bm{i}\sum_{i=1}^d\left(\sum_{\ell=1}^n\left(\cfrac{a_{\ell i}v_{\ell}}{s_{\ell}+r_{\kappa,\ell}}\right)-\cfrac{1}{(2\pi)^d}\int_{g\in T^d} g^*\left(\cfrac{a_{\ell i}v_{\ell}}{s_{\ell}+r_{\kappa,\ell}}\right)dg\right)d\bar{v}_i\\
		-\left(\sum_{\ell=1}^n\left(\cfrac{a_{\ell i}\bar{v}_{\ell}}{s_{\ell}+r_{\kappa,\ell}}\right)-\cfrac{1}{(2\pi)^d}\int_{g\in T^d} g^*\left(\cfrac{a_{\ell i}\bar{v}_{\ell}}{s_{\ell}+r_{\kappa,\ell}}\right)dg\right)dv_i.
		\end{multline*}
		
		Let $g'_{t,s}$ be the complete Kähler metric determined by $\omega'_{t,s}$. Let $\norm{\cdot}'_{t,s}$ be the norm with respect to $g'_{t,s}$. Since $V'_{t,s}$ is the dual vector field of $a'_{t,s}$, we have
		\[
		\norm{V'_{t,s}}'_{t,s}=\norm{a'_{t,s}}'_{t,s}\le 2\sum_{i=1}^d |a'_{t,s,i}|\norm{dv_i}'_{t,s}=O(|v|^{\frac{1}{2}}),
		\]
		as $|v|\to\infty$. Denote by 		
		$\rho:\C^d\to [0,\infty)$ the Riemannian distance function (from the origin) with respect to the metric $g'_{t,s}$. By \cite{GG}, the auxiliary complete metric $g$ defined by 
		\[
		g=\cfrac{g'_{t,s}}{L^2(\rho'(v))}
		\]
		is complete. $V'_{t,s}$ is uniformly bounded with respect to $g$. Moreover, the time-$1$ flow $\phi_{inv,s}$ generated by $V_{t,s}'$ intertwines $\omega_{sm,s}$ and $\omega_{inv,s}$, as desired.
	\end{proof}

    Denote by $\mathbb{T}$ the \textit{tropical semi-field} $\mathbb{T}=\R\cup\{-\infty\}$. Let $c=(c_1,\ldots,c_d)\in\C^d$ be a point away from the hyperplanes in $\mathcal{H}_{\C}$ previously chosen to be the center of the $T^d$-action.  Recall that the $T^d$-orbits (where are regular fibers of $(|v_1-c_1|,\ldots,|v_d-c_d|)$) are Lagrangian with respect to $\omega_{inv,s}$.
    
    \begin{definition}
    \label{def:lagfib_red}
	Let $\mathrm{Log}_t:\C^d\to\mathbb{T}^d$ be the map defined by 
	\[
	\mathrm{Log}_t(v_1,\ldots,v_d)=\left(\log_{t}|v_1-c_1|,\ldots,\log_{t}|v_d-c_d|\right),
	\]
	where $t\gg0$ is a constant. Denote by $\pi_{s}:X_{s}\to \mathbb{T}^d$ the composition $\pi_{s}=\mathrm{Log}_t\circ\phi_{s}$. $\pi_{s}$ is our preferred Lagrangian torus fibration on $X_{s}$.
	\end{definition}
    
	\subsection{Lagrangian torus fibrations on hypertoric varieties and the discriminant loci} 
	
	    \begin{definition}
	    \label{def:lagfib}
		We denote by $\pi:\mathfrak{M}_{u,\lambda}\to B=\R^d\times\mathbb{T}^d$ the map which sends a point $x\in \bar{\mu}_{\R}^{-1}(\frac{s}{2})$ to $\pi(x)=\left(s,\pi_{s}\left([x]\right)\right)$, where $[x]\in X_{s}$ is the $T^d$-orbit of $x$. $\pi$ is a piecewise smooth Lagrangian torus fibration.
		\end{definition}
	
	Let $b=(s,\tau)=(s_1,\ldots,s_d,\tau_1,\ldots,\tau_d)\in B$. For generic values $b$, the fiber $\pi^{-1}(b)\cong T^{2d}$ is a smooth Lagrangian torus. When exactly $k$ components of $\tau$ is $-\infty$, the fiber $\pi^{-1}(b)$ degenerates to a torus $T^{2d-k}$. If $s\in H_{\R,i}$ and $\tau\in\pi_{s}\left(H_{\C,i}\right)$, the fiber $\pi^{-1}(b)$ is a \textit{pinched torus} (i.e. a product of immersed $\bS^2$ and tori) of dimension $2d$. We denote by $\Sigma\subset B$ the set of all points over which the fibers of $\pi$ are singular:
	\[
	\Sigma=\partial B\cup\left(\bigcup_{i=1}^n \{(s,\tau)\in B|s\in H_{\R,i}\text{ and }\tau\in\pi_{s}\left(H_{\C,i}\right)\}\right).
	\] 
	We will call $\Sigma$ the \textit{discriminant loci} of $\pi$ (e.g. Fig \ref{fig:Lag-fib-HT}). Let $B^0=B\setminus\Sigma$. $\pi$ restricts to a $T^{2d}$-bundle over $B^0$, and induces an integral affine structure on $B^0$.
	
	\begin{figure}[h]
		\begin{center}
			\includegraphics[scale=0.5]{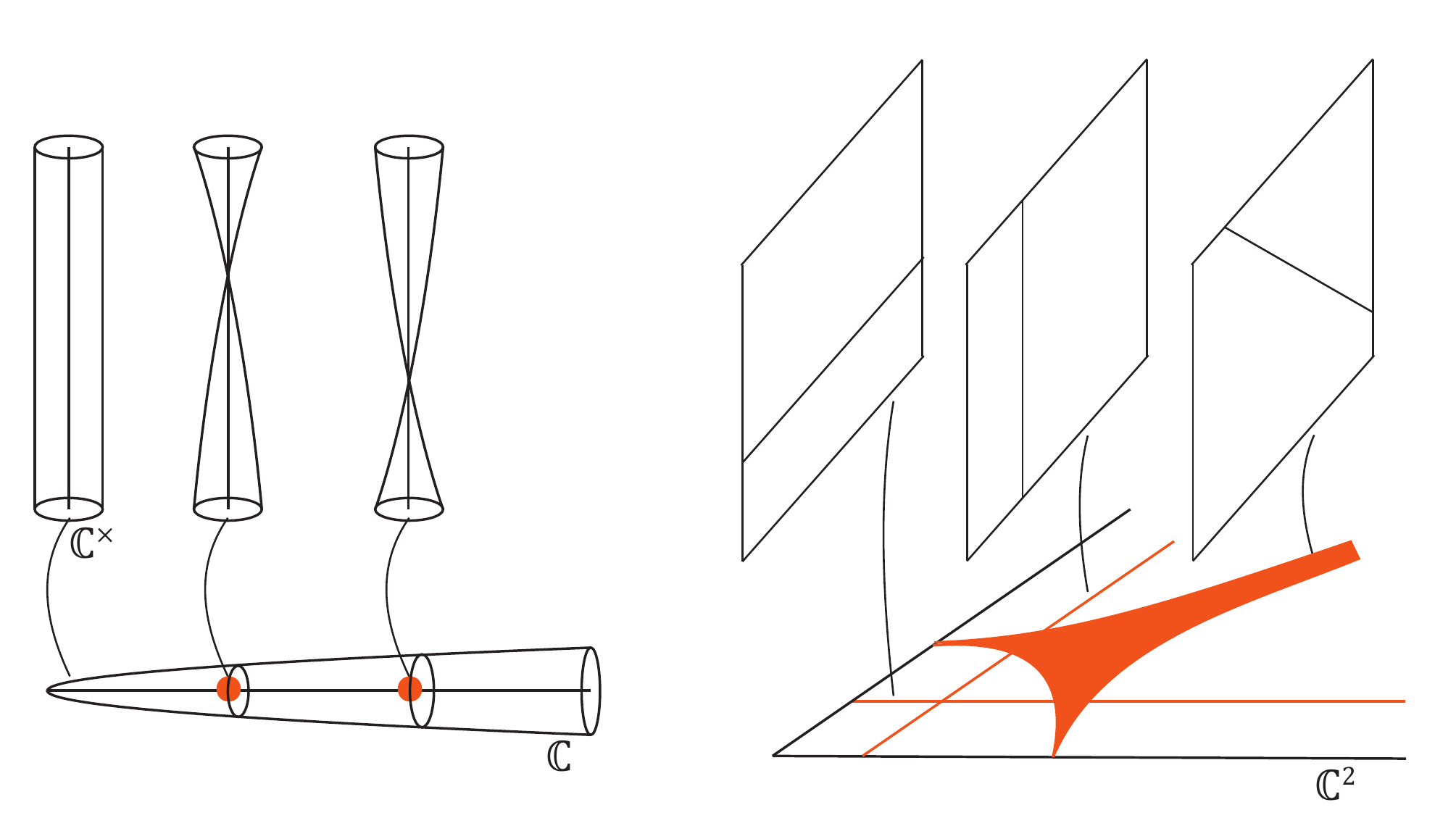}
			\caption{Lagrangian fibrations on $T^*\C\bP^1$ and $T^*\C\bP^2$, where the base are $\R\times\mathbb{T}$ and $\R^2\times\mathbb{T}^2$, respectively. The complex hyperplanes are taken to be in general positions.}
			\label{fig:Lag-fib-HT}
		\end{center}
	\end{figure}
	
	\section{SYZ mirror construction for hypertoric varieties}
	\label{SYZ}
	In this section, we carry out the SYZ mirror construction for smooth hypertoric varieties. We begin by reviewing the SYZ construction.
	
	\subsection{The SYZ mirror construction}
	\label{sec:SYZ}
	Let $\pi: X \to B$ be a proper Lagrangian torus fibration of a compact K\"ahler manifold $(X,\omega)$ of dimension $d$  
	such that the base $B$ is a compact manifold with corners, and the preimage of each codimension-one facet of $B$ is a smooth irreducible divisor denoted by $D_i$ for $1\le i \le m$.
	
	We assume that the regular Lagrangian fibers of $\pi$ are special with respect to a nowhere-vanishing meromorphic volume form $\Omega$ on $X$ whose pole divisor is the boundary divisor $D:=\sum_{i=1}^m D_i$ (and hence $D$ is an anti-canonical divisor).  We denote by $B^0 \subset B$ the complement of the discriminant locus of $\pi$, and we assume that $B^0$ is connected\footnote{When the discriminant locus has codimension-two, $B^0$ is automatically connected.  Although the Lagrangian fibrations on hypertoric varieties that we constructed have codimension-one discriminant loci, $B^0$ is still connected.}.  We denote by $L_b$ a fiber of $\pi$ over $b \in B^0$.
	
	\begin{lemma}[Maslov index of disc classes {\cite[Lemma 3.1]{Auroux07}}] \label{MaslovIndex} 
		For a disc class $\beta \in \pi_2(X,L_b)$ where $b \in B^0$, the Maslov index of $\beta$ is $\mu(\beta)=2[D]\cdot \beta$. 
	\end{lemma}
	
	
	\begin{definition}[Wall \cite{CLL}] \label{def:wall}
		The \textit{wall} $\bm{W}$ of a Lagrangian fibration $\pi:X\to B$ is the set of points $b \in B^0$ 
		such that the fiber $L_b$ bounds nonconstant holomorphic discs with Maslov index $0$.  
	\end{definition}
	
	The complement of $\bm{W}\subset B^0$ consists of several connected components, which we call \textit{chambers}. Over different chambers the Lagrangian fibers behave differently in a Floer-theoretic sense. 
	Away from the wall $\bm{W}$, the fibers are \textit{weakly unobstructed} and the one-pointed open Gromov--Witten invariants are well-defined using the machinery of Fukaya--Oh--Ohta--Ono \cite{FOOO}. 
	
	\begin{definition}[Open Gromov--Witten invariants {\cite{FOOO}}] \label{def:oGW}
		For $b \in B^0 \setminus\bm{W}$ and $\beta \in \pi_2(X,L_b)$, 
		let  $\mathcal{M}_1(\beta;L_b)$ be the moduli space of stable discs with one boundary marked point of class $\beta$, and $[\mathcal{M}_1(\beta;L_b)]^{\mathrm{vir}}$ be the virtual fundamental class of $\mathcal{M}_1(\beta;L_b)$.
		The \textit{open Gromov--Witten invariant} associated to $\beta$ is $n_{\beta}:=\int_{[\mathcal{M}_1(\beta;L_b)]^{\mathrm{vir}}}\mathrm{ev}^*[\mathrm{pt}]^{\mathrm{PD}}$, 
		where $\mathrm{ev}:\mathcal{M}_1(\beta;L_b) \to L_b$ is the evaluation map at the boundary marked point and $[\mathrm{pt}]^{\mathrm{PD}}$ is the Poincar\'e dual of the point class of $L_b$.
	\end{definition}
	We will restrict to disc classes which are transversal to the boundary divisor $D$ when we construct the mirror space (while for the mirror superpotential we need to consider all disc classes).
	
	\begin{definition}[Transversal disc class] \label{def:transversal}
		A disc class $\beta \in \pi_2(X,L_b)$ for $b \in B^0$ is said to be transversal to the boundary divisor $D$, which is denoted as $\beta \pitchfork D$, if it is represented by a map $u$ with $\mathrm{Im}(u) \cap D$ being a finite set of points and the intersections are transversal
	\end{definition}
	
	Due to dimension reason, the open Gromov--Witten invariant $n_\beta$ is nonzero only when the Maslov index $\mu(\beta)=2$. 
	When $\beta$ is transversal to $D$ or when $X$ is semi-Fano, namely $c_1(\alpha) = [D] \cdot \alpha \geq 0$ for all holomorphic sphere classes $\alpha$, the number $n_\beta$ is invariant under small deformation of complex structure and under Lagrangian isotopy in which all Lagrangian submanifolds in the isotopy do not intersect $D$ nor bound nonconstant holomorphic disc of Maslov index less than $2$. 
	
	
	The SYZ mirror construction can be realized as follows \cite{CLL}.  First, the semi-flat mirror $X^\vee_0$  is defined as the space of pairs $(L_b,\nabla)$ 
		where $b \in B^0$ and $\nabla$ is a flat $\U(1)$-connection on the trivial complex line bundle over $L_b$ up to gauge. 
		There is a natural map $\pi^\vee:X^\vee_0\to B^0$ given by forgetting the second coordinate. 
		The semi-flat mirror  $X^\vee_0$ has a canonical complex structure \cite{Leung} 
		and the functions $\mathrm{e}^{-\int_{\beta}\omega}\mathrm{Hol}_{\nabla}(\partial \beta)$ on $X^\vee_0$ for disc classes $\beta \in \pi_2(X,L_b)$ are called semi-flat complex coordinates. 
		Here $\mathrm{Hol}_{\nabla} (\partial \beta)$ denotes the holonomy of the flat $\U(1)$-connection $\nabla$ along $\partial \beta \in \pi_1(L_b)$. 
		
	Then the generating functions of transversal open Gromov--Witten invariants are defined by 
		\begin{equation}
		\mathcal{I}_i(L_b,\nabla) := \sum_{\substack{\beta \in \pi_2(X,L_b) \\ \beta \cdot D_i = 1, \beta\pitchfork D}} n_\beta \exp\left(-\int_{\beta}\omega\right)\mathrm{Hol}_{\nabla}(\partial \beta), 
		\label{eq:gen}
		\end{equation}
		for $1 \le i \le m$, $(L_b, \nabla) \in (\pi^\vee)^{-1}(B^0\setminus \bm{W})$.  They serve as quantum corrected complex coordinates. 
		The function $\mathcal{I}_i$ can be written in terms of the semi-flat complex coordinates, and hence they generate a subring $\C[\mathcal{I}_1, \ldots, \mathcal{I}_m]$ in the coordinate ring\footnote{In general we need to use the Novikov ring instead of $\C$ since $\mathcal{I}_i$ could be a formal Laurent series.  In the cases that we study later, $\mathcal{I}_i$ are Laurent polynomials whose coefficients are convergent, and hence the Novikov ring is not necessary.} of $(\pi^\vee)^{-1}(B^0\setminus \bm{W})$.
	\begin{definition} \label{def:SYZ}
		An SYZ mirror of $X$ is the pair $(X^\vee,W)$ where 
		$X^\vee:=\mathrm{Spec} \left(\C[\mathcal{I}_1,\ldots, \mathcal{I}_m] \right)$ and 
		$$W := \sum_{\substack{\beta \in \pi_2(X,L_b)}} n_\beta \exp\left(-\int_{\beta}\omega\right)\mathrm{Hol}_{\nabla}(\partial \beta). $$
		
		Moreover, $X^\vee$ is called to be an SYZ mirror of $X-D$.  	
	\end{definition}
	
	\begin{remark} \label{rmk:defect}
		In general the mirror space $X^\vee$ defined in this way, which only uses the generating functions of stable discs emanated from boundary divisors, is always affine and can be singular.  The reason is that our construction ignores the local holomorphic functions living on the intermediate chambers in the base and only take the coordinate functions into account.
		
		Indeed for most hypertoric varieties this is the case. A resolution is necessary, and this will be carried out in Section \ref{sec:mirror}. The derived category is expected to be independent of the choice of a resolution. On the other hand, the Lagrangian fibration $\pi$ on $X$ indeed canonically fixes the resolution if we look more closely into Lagrangian Floer theory of the immersed fibers and glue in their formal deformation spaces.  In this paper we will perform the resolution by assuming some combinatorial rules resulting from Lagrangian Floer theory.			
	\end{remark}
	
	\begin{remark}
	Note that $W$ is a sum over all disc classes which are not necessarily transversal.  If $X$ is semi-Fano, then every stable holomorphic disc class of Maslov index $2$ is of the form $\beta+\alpha$ where $\beta$ is transversal with $\mu(\beta)=2$, and $\alpha \in H_2(X)$ with $c_1(\alpha)=0$.  Hence it takes the form $W=\sum_{i=1}^m a_i \mathcal{I}_i$ where $a_i$ are certain series in K\"ahler parameters.  If $X$ is not semi-Fano, then some algebraic manipulation is necessary to write $W$ as a series in $\mathcal{I}_i$ over the Novikov ring.  In this paper we deal with $X-D$ and hence do not concern about $W$.
	\end{remark}
	
			 \subsection{Maslov index $0$ holomorphic discs and walls} \label{sec:wall}
	Let $c=(c_1,\ldots,c_d)\in (\mathfrak{t}_{\C}^d)^*$ be as in Definition \ref{def:lagfib_red}. Denote by $D^-_{i}$ the divisor
	\begin{equation}
	\label{D-}
D^-_{i}=\{[z,w]\in\mathfrak{M}_{u,\lambda} |z_iw_i=c_i\},
	\end{equation}
	  and set $D^-=\sum_{i=1}^d D^-_{i}$. We will assume the isotopies $\phi_{s}$ in Lemma \ref{lemma:moser} preserves $D^-$. This can be achieved by modifying $\phi_{s}$ using the construction in {\cite[Lemma B.2]{AAK}.
		
		\begin{lemma}[Maslov index formula]
			\label{lemma:maslov}
			Let $L_b=\pi^{-1}(b)$ be the fiber of $\pi:\mathfrak{M}_{u,\lambda}\to B$ over $b\in B^0$. For any disc class $\beta\in \pi_2(\mathfrak{M}_{u,\lambda}, L_b)$, the Maslov index $\mu(\beta)$ is twice the algebraic intersection number $\beta\cdot [D^-]$.
		\end{lemma}
		
		\begin{proof}
			Let $\Omega$ be the meromorphic volume form on $\mathfrak{M}_{u,\lambda}$ with pole divisor $D^-$ defined by 
			\[
			\Omega=\cfrac{\bigwedge_{i=1}^d dz_i\wedge dw_i}{\prod_{i=1}^d z_iw_i-c_i}.
			\]
		Let $b=(s,\tau)$. If $s\notin H_{\R,i}$ for all $i$, the $T^n/K$-action on the level set $\bar{\mu}_{\R}^{-1}\left(\frac{s}{2}\right)$ containing $L_b$ is free, and hence $\bar{\mu}_{\R}^{-1}\left(\frac{s}{2}\right)$ is a trivial $T^d$-bundle over $\C^d$. From Lemma \ref{lemma:moser}, we have a one parameter family $(\phi_{s,t})_{t\in[0,1+\kappa]}$ of homeomorphisms taking the projection $\bar{\mu}_{\C}(L_b)\subset\C^d$ of $L_b$ to a standard product torus centered at the point $c$. We can lift $(\phi_{s,t})_{t\in[0,1+\kappa]}$ to $\bar{\mu}_{\R}^{-1}(\frac{s}{2})$ by defining it to be fiber-wise constant and extend it to a one parameter family of homeomorphisms of $(\Phi_{b,t})_{t\in[0,1+\kappa]}$ of $\mathfrak{M}_{u,\lambda}$. If $s\in H_{\R,i}$, we can isotope $L_b$ to a nearby smooth fiber $L_{b'}$ contained in a level set $\bar{\mu}_{\R}^{-1}(\frac{s'}{2})$ with $s'\notin H_{\R,i}$ for all $i$, and then define $(\Phi_{b,t})_{t\in[0,1+\kappa]}$ by pre-composing $(\Phi_{b',t})_{t\in[0,1+\kappa]}$ with this isotopy.
The phase function $\arg(\Omega|_{\Phi_{b,1+\kappa}(L_b)}):\Phi_{b,1+\kappa}(L_b)\to \bS^1$ is identically zero since $\Phi_{b,1+\kappa}(L_b)$ is a special Lagrangian in $\mathfrak{M}_{u,\lambda}\setminus D^-$. This means the map $\arg(\Omega|_{L_b})_*:\pi_1(L_b)\to \pi_1(\bS^1)=\Z$ induced by $\arg(\Omega|_{L_b}):L_b\to \bS^1$ is trivial, and hence the Maslov class of $L_b$ vanishes in $\mathfrak{M}_{u,\lambda}\setminus D^{-}$, i.e. $\arg(\Omega|_{L_b})$ lifts to a real-valued function. It is then a well known fact (see \cite[Lemma 3.1]{Auroux07} and \cite{AAK}) that $\mu(\beta)=2\beta\cdot D^{-}$.
		\end{proof}
		
		\begin{prop}
			\label{prop:walls}
			The set of points $b\in B^0$ such that the fiber $L_b$ bound nontrivial holomorphic discs of Maslov index $0$ is the union $\bigcup_{i=1}^n W_i$, where $W_i$ is defined by  
			\[
			W_i=\{(s,\tau)\in B^0|\tau\in \pi_{s}(H_{\C,i})\}.
			\]
		\end{prop}
We will refer to $W_i$ as the walls of Lagrangian torus fibration $\pi:\mathfrak{M}_{u,\lambda}\to B$.
		
		\begin{proof}
			Let $L_b$ be the fiber of $\pi$ over $b=(s,\tau)\in B^0$. Then, $L_b$ is contained in the level set $\bar{\mu}_{\R}^{-1}(\frac{s}{2})$. Let $u:(D^2,\partial D^2)\to (\mathfrak{M}_{u,\lambda},L_b)$ be a holomorphic disc with boundary in $L_b$ representing a disc class $\beta\in\pi_2(\mathfrak{M}_{u,\lambda},L_b)$ with $\mu(\beta)=0$. Denote by $L_{red}$ the projection of $L_b$ to $\C^d$ via $\bar{\mu}_{\C}$. $L_{red}$ is a Lagrangian torus with respect to $\omega_{red,s}$, and its projection to the $i^{\mathrm{th}}$ component is a loop around $c_i$. The image of the holomorphic disc $\bar{\mu}_{\C}\circ u:(D^2,\partial D^2)\to (\C^d,L_{red})$ is contained in $\C^d\setminus\{c\}$ by Proposition \ref{lemma:maslov}. By maximal principle, $\bar{\mu}_{\C}\circ u$ is necessarily constant. This means the image of $u$ is contained in a fiber $\bar{\mu}_{\C}^{-1}(v_0)$ for some $v_0\in \C^d$.
			
			If $b\notin W_i$ for all $i$, then we have $v_0\notin H_{\C,i}$ for all $i$. In this case,  $\bar{\mu}_{\C}^{-1}(v_0)\cong (\C^{\times})^d$, while $\bar{\mu}_{\C}^{-1}(v_0)\medcap L_b=T^d$ is a product torus in $(\C^{\times})^d$ centered at the origin. Maximal principle then implies that $u$ is necessarily constant.
			
			 On the other hand, let $I\subset\{1,\ldots,n\}$ be the set of indices such that $b\in W_i$ and suppose $I\ne\emptyset$. Then, we can have $v_0\in H_{\C,i}$ for $i\in I'$ where $I'\subset I$ is a nonempty subset. In which case, $\bar{\mu}_{\C}^{-1}(v_0)\cong(\C\cup_0\C)^{|I'|}\times(\C^{\times})^{d-|I'|}$.  $\bar{\mu}_{\C}^{-1}(v_0)\medcap L_b=T^d$ is a product torus in $(\C\medcup_0\C)^{|I'|}\times(\C^{\times})^{d-{|I'|}}$ such that each $\C\medcup_0\C$ contains a $\bS^1$-component of $T^d$ in one of the irreducible components (depending on the signs of the corresponding components of $s$). It is then easy to see that $\bar{\mu}_{\C}^{-1}(v_0)\medcap L_b$ bounds exactly ${|I'|}$ nonconstant holomorphic discs (and all their multiple covers) of Maslov index $0$.
		\end{proof}
			
		\begin{remark} The construction of $(\phi_{s})_{s\in (\mathfrak{t}^d)^*}$ in Lemma \ref{lemma:moser} gives us an one-parameter family of homeomorphisms of $B^0$ taking each $W_i$ to $(\R^d\setminus H_{\R,i})\times\mathrm{Log}_t(H_{\C,i})$, where $\mathrm{Log}_t(H_{\C,i})$ is a amoeba that retracts to a tropical hyperplane in $\mathbb{T}^d$ as $t\to\infty$. Since we only need the wall and chamber structure on $B^0$ for the mirror construction, which is purely combinatorial, we will simply illustrate each $W_i$ as a tropical hyperplane in $\mathbb{T}^d$ (see Fig \ref{fig:chambers}). 	
		\end{remark}
		
		\subsection{Chambers and simply connected affine charts.} \label{sec:chambers}
		Let $H$ be a tropical hyperplane in $\mathbb{T}^d$ defined by the tropical polynomial $\max\{\tau_{i_1},\ldots,\tau_{i_m},a\}$. $H$ divides $\mathbb{T}^d$ into tropical chambers each of which a monomial of the defining equation attains maximum. We label the chamber where the constant $a$ attains maximum by $0$, and the chamber where the monomial $\tau_{i}\in\{\tau_{i_1},\ldots,\tau_{i_m}\}$ attains maximum by $i$. Using this convention, we can label the chambers given by a simple arrangement of tropical hyperplanes $\{H_i\}_{i=1}^n$ by $n$-tuples $\bm{h}=(h_1,\ldots,h_n)$, where $h_i\in\{0,\ldots,d\}$ indicates the position of the chamber relative to $H_i$.
		
		Let $\mathcal{H}=\{H_i\}_{i=1}^n$ be the arrangement of tropical hyperplanes $H_i$, where $H_i$ is the tropical limit of $\mathrm{Log}_t(H_{\C,i})$. We can choose $\lambda_{\C}$ such that for $\ell=d+1,\ldots,n$, $|b_{\ell}|$ (in the expression $v_{\ell}=\sum_{i=1}^d a_{{\ell,i}}v_i+b_{\ell})$ are distinct powers of $t$, making $\mathcal{H}$ simple (i.e. every subset of $k$ tropical hyperplanes with nonempty intersection intersects in codimension $k$). We will denote by $\mathcal{C}_{\bm{h}}$ both the tropical chambers and their preimages in $B^0$. This shall not cause any confusion. Notice that the wall and chamber structure on $B^0$ depend on the choice of $\lambda_{\C}$.

		\begin{figure}[h]
			\begin{center}
				\label{fig:chambers}
				\includegraphics[scale=0.75]{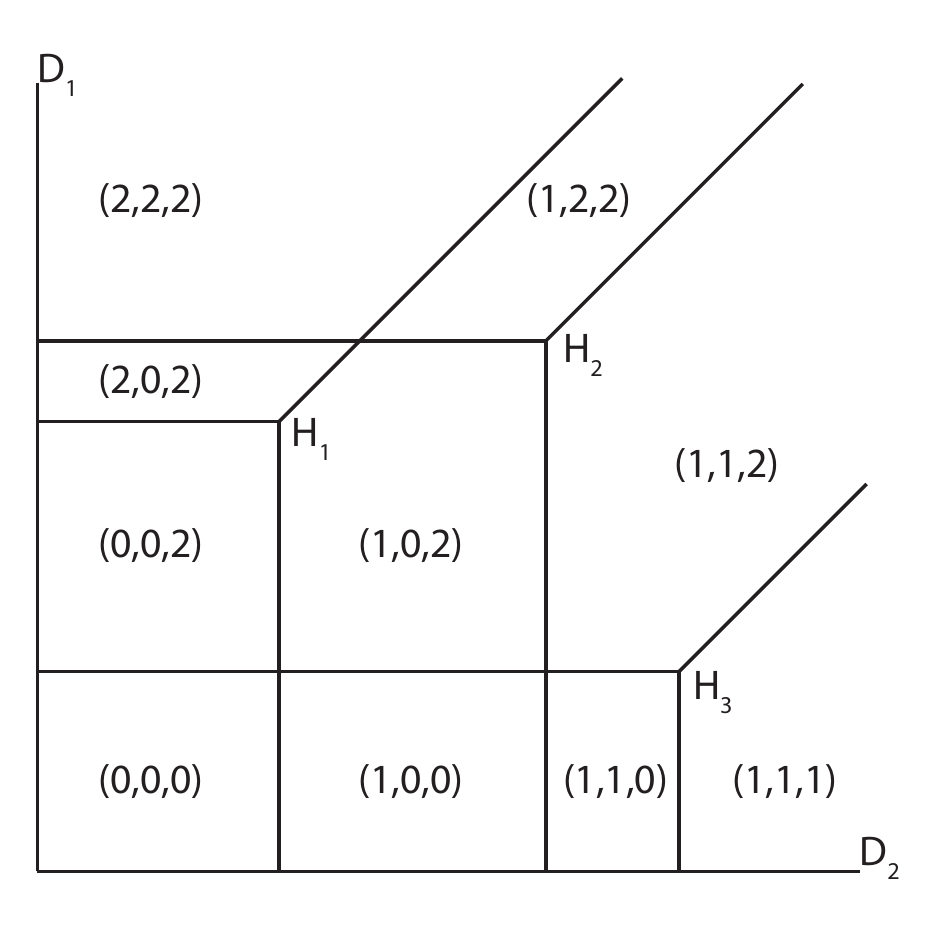}
				\caption{Tropical hyperplane arrangement and chambers}
			\end{center}
		\end{figure}
		
		Let $\sigma$ be a sign vector. We cover $B^0$ by simply connected affine charts $B^0_{\sigma}$ defined by
		\[
		B^0_{\sigma}=\{(s,\tau)\in B^0|s\in H_{\R,i}^{\sigma(i)} \text{ if } \tau\in \pi_s(H_{\C,i})\text{ and }s\in \R^d \text{ if } \tau\notin \pi_s(H_{\C,i})\text{ for all }i\}.
		\]
		
		\subsection{Effective disc classes of Maslov index $2$.} \label{sec:disc2}
		Let $b\in B^0_{\sigma}$ and assume $b$ is inside a chamber $\mathcal{C}_{\bm{h}}$. In particular, this means $b\notin W_i$ for all $i$. Let  $\beta^-_{1},\ldots,\beta^-_{d}\in\pi_2(\mathfrak{M}_{u,\lambda},L_{b})$ be disc classes given by primitive cycles $\gamma_{\sigma,1},\ldots,\gamma_{\sigma,d}\in H_1(L_b,\Z)$ such that $\gamma_{\sigma,i}$ vanishes in the singular fibers over $D^-_{i}$, and let  $\alpha_1,\ldots,\alpha_n\in\pi_2(\mathfrak{M}_{u,\lambda},L_{b})$ be disc classes given by primitive cycles $\gamma_{\sigma,d+1},\ldots,\gamma_{\sigma,d+n}\in H_1(L_b,\Z)$ such that $\gamma_{\sigma,d+i}$ vanishes in the fibers over the interior discriminant locus $\{(s,\tau)\in B|s\in H_{\R,i}\text{ and }\tau\in\pi_{s}\left(H_{\C,i}\right)\}$. When $b\in W_i$, $\alpha_i$ is the Maslov index $0$ disc class described in Proposition \ref{prop:walls}.
		
		We now classify the effective disc classes $\beta\in\pi_2(\mathfrak{M}_{u,\lambda},L_{b})$ of Maslov index $2$.
		
		\begin{prop}
			\label{prop:maslov2}
            The effective disc classes $\beta\in\pi_2(\mathfrak{M}_{u,\lambda},L_b)$ of Maslov index $2$ are of the following form:
			\begin{equation}
			\label{effdisc} \beta=\beta^-_{j}+\delta_1\alpha_{j_1}+\ldots+\delta_N\alpha_{j_N}, \quad j=1,\ldots,d,
			\end{equation}
			where $\delta_k\in\{0,1\}$, and $j_1,\ldots,j_N\in\{1,\ldots,n\}$ are the set of indices such that $h_{j_k}=j$. This means the projections of holomorphic discs of class $\beta$ in $B$ cross the walls $W_{j_1},\ldots,W_{j_N}$.
		\end{prop}
		
		\begin{proof}
		Let $u:(D^2,\partial D^2)\to (\mathfrak{M}_{u,\lambda},L_b)$ be a holomorphic disc of Maslov index $2$. For $i=1,\ldots,n$, denote by $\mathcal{Z}_i$ and $\mathcal{W}_i$ the the divisors
		\[
		 \mathcal{Z}_i=\{[z,w] \in\mathfrak{M}_{u,\lambda}| z_i=0\},
		\]
		and
		\[
		 \mathcal{W}_i=\{[z,w] \in\mathfrak{M}_{u,\lambda}| w_i=0\}.
		\]
		Let $\bm{h}(j)=\{j_1,\ldots,j_N\}$. By Proposition \ref{lemma:maslov} and positivity of intersection, $u$ intersects exactly one divisor $D^-_{j}$ with multiplicity $1$. Thus, $u$ cannot intersect both $\mathcal{Z}_i$ and $\mathcal{W}_i$ for $i\in\bm{h}(j)$ by a winding number argument. For a splitting $I^+\coprod I^-=\bm{h}(j)$ of $\bm{h}(j)$, we define an open subset $U_{(I^+,I^-)}\subset\mathfrak{M}_{u,\lambda}$ by
		\begin{equation}
		\label{conic_bundle}
		U_{(I^+,I^-)}=\{[z,w]\in\mathfrak{M}_{u,\lambda}|z_i\ne0 \text{ if } i\in \{1,\ldots,n\}\setminus I^-; w_i\ne 0\text{ if }  i\in I^-\}.
		\end{equation}
		We have $L_b\in U_{(I^+,I^-)}$ for all splittings $(I^+,I^-)$, and $u(D^2)\subset U_{(I^+,I^-)}$ for exactly one $(I^+,I^-)$ since $b\notin W_i$ for all $i$. Note that each $U_{(I^+,I^-)}$ is biholomorphic to the trivial $(\C^{\times})^d$-bundle over $\C^d$. Let $(v_1,\ldots,v_d,\nu_1,\ldots,\nu_d)$ be the complex coordinates on $U_{(I^+,I^-)}$ with $v_i=z_iw_i$ the base coordinates and $\nu_i$ the fiber coordinates. Assume $u(D^2)\subset U_{(I^+,I^-)}$ and write $u:(D^2,\partial D^2)\to (U_{(I^+,I^-)},L_b)$ as
		\[
		u(\zeta)=(v_1(\zeta),\ldots,v_d(\zeta),\nu_1(\zeta),\ldots,\nu_d(\zeta)).
		\]
        By maximal principle, only the $v_j$-component of $u$ is nonconstant. The $v_j$-component of $u$ is unique up to reparametrization. This means all holomorphic discs $u$ of Maslov index $2$ with $u(D^2)\subset U_{(I^+,I^-)}$ for a splitting $(I^+,I^-)$ represent the same disc class in $\pi_2(\mathfrak{M}_{u,\lambda},L_{b})$, which we denote by $\beta_{(I^+,I^-)}$. 
        
        For $i\in I$, set $\mathrm{sgn}(i)=+$ if $i\in I^+$ and $\mathrm{sgn}(i)=-$ if $i\in I^-$. We claim that $\beta_{(I^+,I^-)}=\beta^-_{j}+\delta_1\alpha_{j_1}+\ldots+\delta_N\alpha_{j_N}$, where $\delta_k=1$ if $\mathrm{sgn}(j_k)\ne \sigma(j_k)$, and $\delta_k=0$ if $\mathrm{sgn}(j_k)= \sigma(j_k)$. Since $\mathfrak{M}_{u,\lambda}$ is simply connected (see \cite[Theorem 6.7]{BD}), the following long exact sequence
        \[
        \cdots\to\pi_2(L_b)=0\to\pi_2(\mathfrak{M}_{u,\lambda})\cong H_2(\mathfrak{M}_{u,\lambda};\Z)=\mathfrak{k}_{\Z}\hookrightarrow\pi_2(\mathfrak{M}_{u,\lambda},L_b)\to\pi_1(L_b)\to\pi_1(\mathfrak{M}_{u,\lambda})=0\to\cdots
        \]
        shows that $\pi_2(\mathfrak{M}_{u,\lambda},L_b)$ is generated by the disc classes $\beta_1,\ldots,\beta_d,\alpha_1,\ldots,\alpha_n$. This combined with the intersection numbers of these generators with the divisors proves our claim.   
		\end{proof}		

\subsection{Regularity and open Gromov--Witten invariants}
\label{regularity}
		We now prove regularity of the disc classes in (\ref{effdisc}) and compute relevant open Gromov--Witten invariants necessary for the mirror construction. Our strategies of proofs are similar to that of Lemma 7 and Corollary 8 in \cite{Auroux15}.
		
		Let $u:(D^2,\partial D^2)\to (\mathfrak{M}_{u,\lambda},L_b)$ be a holomorphic disc. Denote by $(\mathcal{E},\mathcal{F})$ the sheaf of holomorphic sections of $E=u^*T\mathfrak{M}_{u,\lambda}$ with boundary values in $F=(u|_{\partial D^2})^*TL_b$. Denote by $\mathcal{A}^0(E,F)$ the sheaf of smooth sections of $E$ with boundary values in $F$, and $\mathcal{A}^{(0,1)}(E)$ the sheaf of smooth $E$-valued $(0,1)$-forms. 
		
		\begin{lemma}{\cite[Lemma 6.2]{CO}}
			The sequence
			\begin{equation}
			\label{ses4}
			0\longrightarrow (\mathcal{E},\mathcal{F})\longrightarrow \mathcal{A}^0(E,F)\overset{\bar{\partial}}\longrightarrow\mathcal{A}^{(0,1)}(E)\longrightarrow 0 
			\end{equation} 
			defines a fine resolution of $(\mathcal{E},\mathcal{F})$.
		\end{lemma}
		
		\begin{prop}
			\label{prop:regularity}
			The holomorphic discs representing classes in (\ref{effdisc}) are Fredholm regular, i.e. its linearization $\bar{\partial}$ is surjective.
		\end{prop}
		
		\begin{proof}
		Let $u:(D^2,\partial D^2)\to (\mathfrak{M}_{u,\lambda},L_b)$ be a holomorphic disc of class $\beta$ in (\ref{effdisc}) in . Denote by $u_{red}$ the composition $\bar{\mu}_{\C}\circ u:(D^2,\partial D^2)\to (\C^d,L_{red})$, where $L_{red}=\bar{\mu}_{\C}(L_b)$. Let $\mathcal{L}_{\R}$ and $\mathcal{L}_{\C}$ be the real and complex spans of the vector fields generating the $T^n/K$-action. Suppose $\beta\cdot[D^-_j]=1$, then, as noted in the proof of Proposition \ref{prop:maslov2}, both $L_b$ and the image of u are contained in an open set $U_{(I^+,I^-)}$ (see (\ref{conic_bundle})) for a splitting $(I^+,I^-)$ of $\bm{h}(j)$. The $T^n/K$-action is free on $U_{(I^+,I^-)}$, and thus we have the following short exact sequences:
		\begin{equation}
		\label{SES1}
		0\longrightarrow\mathcal{L}_{\C}\longrightarrow T\mathfrak{M}_{u,\lambda}\longrightarrow\bar{\mu}_{\C}^*T\C^d\longrightarrow 0,
		\end{equation}
		\begin{equation}
			\label{SES2}
			0\longrightarrow\mathcal{L}_{\R}\longrightarrow TL_b\longrightarrow\bar{\mu}_{\C}^*TL_{red}\longrightarrow 0,
		\end{equation}
		in $U_{(I^+,I^-)}$.	Pulling back the exact sequences above via $u$, we find that $E$ admits a trivial holomorphic subbundle $u^*\mathcal{L}_{\C}$, with a trivial real subbundle $(u|_{\partial D^2})^*\mathcal{L}_{\R}\subset F$ on the boundary.
			Since the $\bar{\partial}$-operator for complex-valued functions on the unit disc with trivial real boundary condition on the boundary circle is surjective, the surjectivity of $\bar{\partial}$ on sections of $E$ with boundary conditions $F$ is then equivalent to the surjectivity of $\bar{\partial}$ on the quotient bundle $E/u^*\mathcal{L}_{\C}=u_{red}^*T\C^d$ with boundary conditions $F/(u|_{\partial D^2})^*\mathcal{L}_{\R}=(u_{red}|_{\partial D^2})^*TL_{red}$. Since only the $j^{\mathrm{th}}$ component of $u_{red}$ is nonconstant, the surjectivity of $\bar{\partial}$ reduces to a one-dimensional Riemann-Hilbert problem which then follows from Theorem II and III in \cite{Oh}.
		\end{proof}

		\begin{prop}
			\label{prop:gw}
			With the notations as in Proposition \ref{prop:maslov2}, we have
			\[
			n_{\beta}=\left\{
			\begin{array}{ll}
			1  & \mbox{for } \beta=\beta^-_{j}+\delta_1\alpha_{j_1}+\ldots+\delta_N\alpha_{j_N},\\
			0 & \mbox{otherwise.}
			\end{array}
			\right.
			\]
		\end{prop}
		
		\begin{proof}
			Due to dimension reason, we have $n_{\beta}=0$ for $\mu(\beta)\ne 2$. Suppose $\beta$ is an effective disc class with $\mu(\beta)=2$, intersecting the divisor $D^-_{j}$. Denote by $p\in\partial D^2$ be the unique boundary marked point on the unit disc $D^2$. Let $L_{red}=\bar{\mu}_{\C}(L_b)\subset\C^d$, and let $\bar{\beta}=(\bar{\mu}_{\C})_*\beta\in\pi_2(\C^d,L_{red})$. 
            Denote by $\bar{D}^-_{i}$ the divisor $\{(v_1,\ldots,v_d)\in\C^d|v_i=c_i\}$. We have $\bar{\beta}\cdot [\bar{D}^-_{j}]= 1$, and $\bar{\beta}\cdot[\bar{D}^-_{i}]=0$ for $i\ne j$.	
			Let's first consider the moduli space $\mathcal{M}_1(L_{red},\bar{\beta})$. By maximal principle, for any $[\bar{u}]\in\mathcal{M}_1(L_{red},\bar{\beta})$, all but the $j^{\mathrm{th}}$ component of $\bar{u}$ are constant, and the $j^{\mathrm{th}}$ component of $\bar{u}$ is unique up to automorphisms of $D^2$ fixing $p$. Thus, for each $q\in L_{red}$, there exists a unique $[\bar{u}]\in\mathcal{M}_1(L_{red},\bar{\beta})$ with $\bar{u}(p)=q$. Moreover, the map $\mathrm{ev}:\mathcal{M}_1(L_{red},\bar{\beta})\to L_{red}$ given by evaluation at the boundary marked point is a diffeomorphism. Now, consider the projection $\mathcal{M}_1(L_b,\beta)\to\mathcal{M}_1(L_{red},\bar{\beta})$ given by post-composing holomorphic discs $u:(D^2,\partial D^2)\to (\mathfrak{M}_{u,\lambda},L_b)$ with $\bar{\mu}_{\C}$. We will show momentarily that for any given $[\bar{u}]\in\mathcal{M}_1(L_{red},\bar{\beta})$, and a lift $\tilde{q}\in L_b$ of $q$, there exist a unique $[u]\in \mathcal{M}_1(L_b,\beta)$ with $\bar{\mu}_{\C}\circ u=\bar{u}$ and $u(p)=\tilde{q}$. Any holomorphic disc in ${M}_1(L_b,\beta)$ has its image is contained in $U_{(I^+,I^-)}$ for a splitting $(I^+,I^-)$ of $\bm{h}(j)$. Recall that $U_{(I^+,I^-)}$ is biholomorphic to the trivial $(\C^{\times})^d$-bundle over $\C^d$. Denote by $(v_1,\ldots,v_d,\nu_1,\ldots,\nu_d)$ the complex coordinates on this open set with $v_1,\ldots,v_d$ being the base coordinates and $\nu_1,\ldots,\nu_d$ being the fiber coordinates. Write $\tilde{q}=(\tilde{q}_1,\ldots,\tilde{q}_{2d})$. We define the lift of $\bar{u}$ to be the holomorphic disc $u:(D^2,\partial D^2)\to (U_{(I^+,I^-)},L_b)$ defined by
			\[ 
			u(\zeta)=(\bar{u}(\zeta),\tilde{q}_{d+1},\ldots,\tilde{q}_{2d}).
			\] 
			We have a free $T^d$-action on $\mathcal{M}_1(L_b,\beta)$ given by composing holomorphic discs $[u]\in\mathcal{M}_1(L_b,\beta)$ with the $T^d$-action on $\mathfrak{M}_{u,\lambda}$. The orbits of this action are exactly the fibers of $\mathcal{M}_1(L_b,\beta)\to\mathcal{M}_1(L_{red},\bar{\beta})$. Therefore, $\mathcal{M}_1(L_b,\beta)\to\mathcal{M}_1(L_{red},\bar{\beta})$ is a $T^d$-bundle. Since the evaluation map $\mathrm{ev}:\mathcal{M}_1(L_b,\beta)\to L_b$ is $T^d$-equivariant, it is again a diffeomorphism, i.e. it is of degree $\pm 1$.
			
		    As for the orientations of $\mathcal{M}_1(L_b,\beta)$, recall that a spin structure on $L_b$ determines an orientation on $\mathcal{M}_1(L_b,\beta)$ (see \cite[Chapter 8]{FOOO}). Since $L_{red}$ is isotopic to the standard product torus in $\C^d$, we can choose the standard spin structure on $L_{red}$ such that $\mathrm{ev}:\mathcal{M}_1(L_{red},\bar{\beta})\to L_{red}$ is orientation-preserving. We choose the spin structure on $L_b$ to be standard along the $T^d$-orbits and consistent under the splitting (\ref{SES2}) with the spin structure previously chosen on $L_{red}$. Then, with the induced orientation on $\mathcal{M}_1(L_b,\beta)$, the evaluation map $\mathrm{ev}:\mathcal{M}_1(L_b,\beta)\to L_b$ is orientation-preserving, i.e. it is of degree $1$.  
		\end{proof}
	
		\begin{prop}
		\label{prop:unobstructed}
		$L_b$ is weakly unobstructed.
	\end{prop}
	\begin{proof}
		Due to degree reason, only stable holomorphic discs of Maslov index less than or equal to $2$ can contribute to $\fm_0^b$.  In our case there is no stable discs with negative Maslov index (Prop. \ref{lemma:maslov}).  Thus the only discs with Maslov index less than $2$ are the constant ones, which are not stable since there is only one output marking.
		For an effective disc class $\beta$ of Maslov index $2$, the evaluation map at the boundary marked point gives a diffeomorphism  $\mathcal{M}_1(L_b,\beta) \to L_b$ (Prop. \ref{prop:gw}).  Hence $\fm_0^b$, which is the sum over $\beta$ of $\mathrm{ev}_* [\mathcal{M}_1(L_b,\beta)]$ weighted by $T^{-\int_\beta\omega}$ (where $T$ is the formal Novikov parameter), is proportional to the fundamental class of $L_b$.
	\end{proof}

		\subsection{Partial compactifications of hypertoric varieties.} \label{sec:cptfy}
		Our idea of constructing the mirror $\mathfrak{M}_{u,\lambda}^{\vee}$ is to construct coordinate functions of $\mathfrak{M}_{u,\lambda}^{\vee}$ by counting holomorphic discs emanating from boundary divisors of ${\mathfrak{M}_{u,\lambda}}$. The problem is that in our situation, $B$ has only $d$ codimension-one boundary, while we need $2d$ coordinate functions. To resolve this, one may consider counting holomorphic cylinders (with one boundary component on $L$ and the other asymptotic to infinity), which requires the extra work of defining rigorously the corresponding Gromov--Witten invariants. Another way is to consider a partial compactification of ${\mathfrak{M}_{u,\lambda}}$ by adding \textit{divisors at infinity} and count the additional holomorphic Maslov index $2$ discs emanated from these divisors. We will use the second approach in this paper. This method was used in \cite{CLL}, \cite{AAK} to construct mirrors of Calabi-Yau toric varieties, and blow-ups of toric varieties along a hypersurface. 
				
		Recall from Remark \ref{rmk:holfib} that the holomorphic moment map $\bar{\mu}_{\C}:\mathfrak{M}_{u,\lambda}\to\C^d$ is a holomorphic $(\C^{\times})^d$-fibration. We can partially compactify $\mathfrak{M}_{u,\lambda}$ by extending $\bar{\mu}_{\C}$ to a holomorphic $(\C^{\times})^d$-fibration over $(\mathbb{P}^1)^d$. 
		
		Let $([\zeta_1:\tilde{\zeta_1}],\ldots,[\zeta_n:\tilde{\zeta_n}])$ be the homogeneous coordinates on $(\mathbb{P}^1)^n$. We embed $\C^d$ into $(\mathbb{P}^1)^n$ via the map $(v_1,\ldots,v_d)\mapsto ([v_1:1],\ldots,[v_d:1],[v_{d+1}:1],\ldots,[v_{n}:1])$, where $v_{\ell}=\sum_{k=1}^d a_{\ell k}v_k+b_{\ell}$ for $\ell=d+1,\ldots,n$,. Its closure $\overline{\C^d}$ in $(\mathbb{P}^1)^n$ is defined by the following homogeneous polynomials
		\[
		f_{\ell}=\tilde{\zeta_1}\ldots\tilde{\zeta}_d\zeta_{\ell}-\sum_{i=1}^d a_{\ell i} \tilde{\zeta}_1\ldots\zeta_i\ldots\tilde{\zeta}_d\tilde{\zeta}_{\ell}+b_{\ell}\tilde{\zeta}_1\ldots\tilde{\zeta}_d\tilde{\zeta}_{\ell}, \quad \ell=d+1,\ldots,n,
		\]
		and is biholomorphic to $(\CP^1)^d$. The hyperplanes $\{H_{\C,i}\}_{i=1}^n$ extends naturally to divisors $\{\bar{H}_{\C,i}\}_{i=1}^n$ on $\overline{\C^d}$ defined by 
		\[
		\bar{H}_{\C,i}=\{([\zeta_1:\tilde{\zeta_1}],\ldots,[\zeta_n:\tilde{\zeta_n}])\in\overline{\C^d}|\zeta_i=0\}.
		\]
		
		Let $E$ be total space of the rank $2n$ complex vector bundle on $(\mathbb{P}^1)^n$ defined by
		\[
		E=\mathcal{O}(\bar{H}_{\C,1})\oplus\mathcal{O}_1\oplus\ldots\oplus\mathcal{O}(\bar{H}_{\C,n})\oplus\mathcal{O}_n\to (\CP^1)^n,
		\]
		where $\mathcal{O}_i=\mathcal{O}$ are trivial complex line bundles. Denote by $w_i$ the fiber coordinate of $\mathcal{O}_i$, $z_i$ the local coordinate of the $\mathcal{O}(\bar{H}_{\C,i})$ over $U_i=\{\tilde{\zeta}_i\ne 0\}$, and $\tilde{z}_i$ the local coordinate of $\mathcal{O}(\bar{H}_{\C,i})$ over $\tilde{U}_i=\{\zeta_i\ne 0\}$. The gluing between $\mathcal{O}(\bar{H}_{\C,i})|_{U_i}$ and $\mathcal{O}(\bar{H}_{\C,i})|_{\tilde{U}_i}$ is given by $z_i\tilde{\zeta}_i=\zeta_i\tilde{z}_i$. For $i=1,\ldots,d$, let $g_i=z_i\tilde{\zeta}_iw_i-\zeta_i$. Let $V\subset E$ the subvariety defined by the ideal $(f_{d+1},\ldots,f_n,g_1,\ldots,g_n)$. 
		
		We now define a $(\C^{\times})^n$-action on $E$. For $\vec{t}=(t_1,\ldots,t_n)\in(\C^{\times})^n$, let $\vec{t}$ act on $\mathcal{O}(\bar{H}_{\C,i})$ via multiplication by $t_i$ and on $\mathcal{O}_i$ via multiplication by $t_i^{-1}$. Let $\vec{t}$ act trivially on the base $(\mathbb{P}^1)^n$. $V$ is then a $(\C^{\times})^n$-invariant subvariety of $E$. Let $K_{\C}\subset (\C^{\times})^n$, and $\lambda_{\R}:K_{\C}\to \C^{\times}$ be the same as in Definition \ref{def:hypertoric}. Then, the GIT quotient
		\[
		\overline{\mathfrak{M}}_{u,\lambda}=V//_{\lambda_{\R}} K_{\C}
		\]
		is a partial compactification of $\mathfrak{M}_{u,\lambda}$. The embedding $\mathfrak{M}_{u,\lambda}\hookrightarrow\overline{\mathfrak{M}}_{u,\lambda}$ is holomorphic and $(\C^{\times})^n/K_{\C}$-equivariant.
		
		Alternatively, we can construct $\overline{\mathfrak{M}}_{u,\lambda}$ via symplectic reduction. Notice that the subbundles $\mathcal{O}(\bar{H}_{\C,i})\to (\CP^1)^n$ of $E$ are the pullbacks of $\mathcal{O}(1)\to \CP^1$ via the projections $(\CP^1)^n\to \CP^1$ to the $i^{\mathrm{th}}$ component. The sum of pullbacks of Fubini-Study form then defines a Kähler form on the total space of the subbundle $\bigoplus_{i=1}^n\mathcal{O}(\bar{H}_{\C,i})$. Combined with the standard symplectic form on the fibers of $\mathcal{O}_i$, we have a $T^n$-invariant Kähler form $\omega_E$ on $E$.
		We can construct $\overline{\mathfrak{M}}_{u,\lambda}$ as the symplectic reduction of $V$ at level $\lambda_{\R}$ with respect to the action of the maximal torus $K\subset K_{\C}$ and the restriction of $\omega_{E}$ to $V$. This equips $\overline{\mathfrak{M}}_{u,\lambda}$ with a $T^n/K$-invariant Kähler form $\bar{\omega}$.
		
		We can then construct a Lagrangian torus fibration
		\[
		\bar{\pi}:\overline{\mathfrak{M}}_{u,\lambda}\to\bar{B}=\R^d\times(\R\cup\{\pm\infty\})^d
		\] 
	using symplectic reductions as in Section \ref{sec:fib}. The reduced spaces are biholomorphic to $(\mathbb{P}^1)^d$. Since the reduced spaces are now compact, the construction of $\bar{\pi}$ is simple applications of Moser's trick, and hence omitted. The discriminant loci $\bar{\Sigma}$ of $\bar{\pi}$ is the union $\Sigma$ and the new boundaries of $\bar{B}$ at infinity. Notice that we have $\bar{B}\setminus\bar{\Sigma}=B^0\subset B$.
	
	\begin{remark}
If we were to strictly follow the SYZ construction outlined in Section \ref{sec:SYZ}, we could have compactified $\mathfrak{M}_{u,\lambda}$ by compactifying the fiber directions of $E$. However, since the cycles $\gamma_{\sigma,d+1},\ldots,\gamma_{\sigma,2d}\in H_1(L_b;\Z)$ (see Section \ref{sec:disc2}) are monodromy-invariant, the count of holomorphic discs emanated from the these additional divisors would receive no quantum correction. Therefore, it suffices to consider the partial compactification $\overline{\mathfrak{M}}_{u,\lambda}$.
	\end{remark}
		
		We now state the results analogous to Propositions \ref{lemma:maslov}, \ref{prop:walls}, \ref{prop:maslov2}, \ref{prop:regularity}, \ref{prop:unobstructed} and \ref{prop:gw} in order to define the additional generating functions.  We will be brief since the proofs are nearly identical to the previous ones.
				
		Denote by $D^+_{i}$ the divisor given by
		\[ D^+_{i}=\{(\{\tilde{\zeta}_i=0\}\medcap V)//_{\lambda_{\R}} K_{\C}\}.\]
		 Let $D^+:=\sum_{i=1}^d D^+_{i}$, and set $D:=D^-+D^+$. We will assume the isotopies obtained from Moser's tricks leaves $D$ invariant.
		
		\begin{prop}
			\label{prop:cptmaslov}
			Let $L_b$ be the fiber of $\bar{\pi}:\overline{\mathfrak{M}}_{u,\lambda}\to\bar{B}$ over $b\in B^0$. For any disc class $\beta\in \pi_2(\overline{\mathfrak{M}}_{u,\lambda}, L_b)$, the Maslov index $\mu(\beta)$ is equal to twice the algebraic intersection number $\beta\cdot [D]$.
		\end{prop}
		
		\begin{proof}
			We first extend the meromorphic volume form $\Omega$ (see Proposition \ref{lemma:maslov}) on $\mathfrak{M}_{u,\lambda}$ to a meromorphic volume form on $\overline{\mathfrak{M}}_{u,\lambda}$ with  generically simple poles along $D$. Consider the form 
			\[
			\bar{\Omega}=\cfrac{\bigwedge_{i=1}^n d\log\xi_i\wedge d\log w_i}{\prod_{i=1}^n 1-\frac{c_i}{\xi_i}}
			\]
			defined on $U=\{\tilde{\zeta}_i\ne 0,\forall i\}\subset E$. Its restriction to $U\cap V$ descends to $\Omega$ on $\mathfrak{M}_{u,\lambda}$. Let $I\subset\{1,\ldots,n\}$, and set $U_I=\{\tilde{\zeta}_i\ne 0,\forall i\in I\}$, $\tilde{U}_I=\{\zeta_i\ne 0,\forall i\in I\}$. Let $I^-\coprod I^+$ be a splitting of $\{1,\ldots,n\}$. We extend $\bar{\Omega}$ to $W$ by defining it to be
			\[
			\cfrac{(-1)^{\mathrm{sgn}(I^-,I^+)}\left(\bigwedge_{i\in I^-}d\log\xi_i\wedge d\log w_i\right)\left(\bigwedge_{j\in I^+}-d\log \tilde{\xi_j}\wedge d\log w_j\right)}{\left(\prod_{i\in I^-} 1-\frac{c_i}{\xi_i}\right)\left(\prod_{j\in I^+} 1-c_j\tilde{\xi_j}\right)}
			\]
			on $U_{I^-}\cap \tilde{U}_{I^+}$, where $\mathrm{sgn}(I^-,I^+)$ is the sign of the concatenation of $I^-$ and $I^+$ as a permutation. Note that the expression above is simply given by rewriting $\bar{\Omega}$ under the change of coordinates. We denote the extension of $\bar{\Omega}$ to $E$ again by $\bar{\Omega}$. $\bar{\Omega}$ is $(\C^{\times})^n$-invariant, hence its restriction to $V$ descends to a meromorphic volume form on $\overline{\mathfrak{M}}_{u,\lambda}$, which is the extension of $\Omega$. With $\bar{\Omega}$ constructed, the proof then follows from Proposition \ref{lemma:maslov}.
		\end{proof}	
		
		The restriction of the projection $E\to (\CP^1)^n$ to $V$ descends to a holomorphic $(\C^{\times})^d$-fibration $\rho:\overline{\mathfrak{M}}_{u,\lambda}\to (\CP^1)^d$, extending $\bar{\mu}_{\C}:\mathfrak{M}_{u,\lambda}\to\C^d$. We denote by $\rho_0=\bar{\mu}_{\C}:\overline{\mathfrak{M}}_{u,\lambda}\setminus D^+\to\C^d$, and $\rho_{\infty}:\overline{\mathfrak{M}}_{u,\lambda}\setminus D^-\to\C^d$ the restrictions of $\rho$ to the respective domains.
		
		\begin{prop}
			\label{walls2}
			The walls of the Lagrangian torus fibration $\bar{\pi}:\overline{\mathfrak{M}}_{u,\lambda}\to\bar{B}$ are the sets $\{W_i\}_{i=1}^n$ defined in Proposition \ref{prop:walls}.
		\end{prop}
		
		\begin{proof}
		Since $\bar{B}\setminus\bar{\Sigma}=B^0$, any fiber over $B^0$ is contained in $\mathfrak{M}_{u,\lambda}$. By Proposition \ref{prop:cptmaslov}, any Maslov index $0$ holomorphic disc $u:(D^2,\partial D^2)\to (\overline{\mathfrak{M}}_{u,\lambda},L_b)$ is contained in $\mathfrak{M}_{u,\lambda}$. Composing $u$ with $\rho_0$ reduces this to Proposition\ref{prop:walls}.
		\end{proof}
				
		We again denote by $\mathcal{C}_{\bm{h}}$ and $B^0_{\sigma}$ the chambers and simply connected affine charts on $B^0\subset\bar{B}$, respectively. Let's fix a reference point $b\in B^0_{\sigma}$ and assume $b\in \mathcal{C}_{\bm{h}}$. We now classify the effective disc classes $\beta\in\pi_2(\overline{\mathfrak{M}}_{u,\lambda},L_{b})$ with $\mu(\beta)=2$.
	
	We express any vector $v$ in the basis $\{u_1,\ldots,u_d\}$, and denote the corresponding coefficients by $v^{(i)}$.			
		\begin{prop}
			\label{prop:cptmaslov2}
			Denote by $\beta^+_{1},\ldots,\beta^+_{d}\in\pi_2(\overline{\mathfrak{M}}_{u,\lambda},L_b)$ the disc classes given by the cycles $\gamma_{\sigma,1}\ldots,\gamma_{\sigma,d}\in H_1(L_b;\Z)$ (see Section \ref{sec:disc2}) vanishing on $D^+_{1},\ldots,D^+_{d}$. The effective disc classes $\beta\in\pi_2(\overline{\mathfrak{M}}_{u,\lambda},L_b)$ with $\mu(\beta)=2$ are of the form
			\begin{equation}
			\label{effdisc2}
			\beta=\beta^{\pm}_{j}+\delta_1\alpha_{j_1}+\ldots+\delta_N\alpha_{j_N},\quad j=1,\ldots,d,
			\end{equation}
			where $\delta_i\in\{0,1\}$, and $j_1,\ldots,j_N\in\{1,\ldots,n\}$ is the set of indices such that $h_{j_k}=j$ if $\beta$ has the $\beta^-_{j}$ component, and it is the set indices such that $u_{j_k}^{(j)}\ne 0$ and $h_{j_k}\ne j$ if $\beta$ has the $\beta^+_{j}$ component. This means the projections of holomorphic discs of class $\beta$ in $\bar{B}$ cross the walls $W_{j_1},\ldots,W_{j_N}$.
		\end{prop}
		
		\begin{proof}
		By Proposition \ref{prop:cptmaslov}, an effective disc class $\beta\in\pi_2(\overline{\mathfrak{M}}_{u,\lambda},L_b)$ with $\mu(\beta)=2$ must intersects either $D^-$ or $D^+$ with multiplicity $1$. In either case, we can classify the effective disc classes by using local charts as in Proposition \ref{prop:maslov2}.
		\end{proof}
		
		\begin{prop}
			The holomorphic discs representing classes in (\ref{effdisc2}) are Fredholm regular.
		\end{prop}
		
		\begin{proof}
		Let $u:(D^2,\partial D^2)\to (\overline{\mathfrak{M}}_{u,\lambda},L_b)$ be a holomorphic disc representing $\beta$ in (\ref{effdisc2}). By the same argument as in \ref{prop:regularity}, regularity of $u$ is equivalent to regularity of $\rho_0\circ u$ if $\beta\cdot [D^{-}]=1$ and of $\rho_{\infty}\circ u$ if $\beta\cdot [D^{+}]=1$, which is then a one-dimensional Riemann-Hilbert problem and  follows from Theorem II and III of \cite{Oh}.
		\end{proof}		
		
		\begin{prop}
			\label{prop:gw2}
			With the notations as in Proposition \ref{prop:cptmaslov2}, we have
			\[
			n_{\beta}=\left\{
			\begin{array}{ll}
			1  & \mbox{for } \beta=\beta^{\pm}_{j}+\delta_1\alpha_{j_1}+\ldots+\delta_N\alpha_{j_N},\\
			0 & \mbox{otherwise}.
			\end{array}
			\right.
			\]
		\end{prop}
		
		\begin{proof}
			The proof is identical to that of Proposition \ref{prop:gw} except we have $T^d$-bundles $\mathcal{M}_1(L_b,\beta)\to\mathcal{M}_1(\rho_0(L_b),(\rho_0)_*(\beta))$ and $\mathcal{M}_1(L_b,\beta)\to\mathcal{M}_1(\rho_{\infty}(L_b),(\rho_{\infty})_*(\beta))$ depending on whether $\beta\cdot [D^{-}]=1$ or $\beta\cdot [D^{+}]=1$.
		\end{proof}
	
Similar to Proposition \ref{prop:unobstructed}, we have
			\begin{prop}
		$L_b$ is weakly unobstructed.
	\end{prop}
		
		\subsection{Generating functions of open Gromov--Witten invariants and wall-crossing}
		\label{sec:GF}
		
		Denote by $\mathfrak{M}_0^{\vee}$ the semi-flat mirror of $\mathfrak{M}_{u,\lambda}$(see Section \ref{sec:SYZ}). The semi-flat complex coordinates on $\mathfrak{M}_0^{\vee}$ is defined as follows.
		
		For each simply connected affine chart $B^0_{\sigma}$, we have an open subset $\mathfrak{M}_{0,\sigma}^{\vee}=(\pi^\vee)^{-1}(B^0_{\sigma})\subset\mathfrak{M}_0^{\vee}$. For each $\sigma$, we fix a reference point $b_{\sigma}\in B^0_{\sigma}$. Let $\{\gamma_{\sigma,1},\ldots,\gamma_{\sigma,2d}\}\subset H_1(L_{b_{\sigma}})$ be the cycles described in Section \ref{sec:disc2} and note that they form a primitive integer basis. 
		
		\begin{definition}
		\label{def:semiflatcoord}
		The semi-flat complex coordinates on $\mathfrak{M}_0^{\vee}$ is defined locally on the charts $\mathfrak{M}_{0,\sigma}^{\vee}$ by
		\[
		Z_{\sigma,i}(L_{b},\nabla)=\exp\left(-\int_{\Gamma_{\sigma,i}(b)}\bar{\omega}\right)\mathrm{Hol}_{\nabla}(\gamma_{\sigma,i}(b)),\quad i=1,\ldots,2d,
		\]
		where $\gamma_{\sigma,i}(b)\in H_1(L_{b},\Z)$ is the parallel transport of $\gamma_{\sigma,i}$, and $\Gamma_{\sigma,i}(b)$ is the cylinder given by parallel-transporting $\gamma_{\sigma,i}$. 
        \end{definition}
    
    The transition map between the charts $\mathfrak{M}_{0,\sigma}^{\vee}$ and $\mathfrak{M}_{0,\sigma'}^{\vee}$ is given by (exponential of) the integral affine transformation between $B^0_{\sigma}$ and $B^0_{\sigma'}$. 		
		\begin{definition} 
			\label{def:gen}
			The generating functions $\bm{u}_j$ (resp. $\bm{v}_j$) for discs emanated from boundary divisors $D^-_{j}$ (resp. $D^+_{j}$) for $j=1,\ldots,d$, are given by
			\[
			\bm{u}_j(L_b,\nabla)= \sum_{\substack{\beta \in \pi_2(X,L_b) \\ \beta \cdot D^-_{j} = 1, \beta\pitchfork D}} n_\beta \exp\left(-\int_{\beta}\bar{\omega}\right)\mathrm{Hol}_{\nabla}(\partial \beta),
			\]
			\[
			\bm{v}_j(L_b,\nabla)= \sum_{\substack{\beta \in \pi_2(X,L_b) \\ \beta \cdot D^+_{j} = 1, \beta\pitchfork D}} n_\beta \exp\left(-\int_{\beta}\bar{\omega}\right)\mathrm{Hol}_{\nabla}(\partial \beta).
			\]
		\end{definition} 
		
		Let $C_{\sigma,i}=\exp\left(-\int_{\beta^-_{i}}\bar{\omega}\right)$ and $C_{\sigma,d+i}=\exp\left(-\int_{\alpha_i}\bar{\omega}\right)$ for $i=1,\ldots,d$, where $\beta^-_{i},\alpha_i\in H_2(\overline{\mathfrak{M}}_{u,\lambda},L_{b_{\sigma}})$ are as descried in Section \ref{sec:disc2}. Since the cycles $\gamma_{\sigma,d+1},\ldots,\gamma_{\sigma,2d}\in H_1(L_b;\Z)$ are monodromy-invariant, we have $C_{\sigma,d+i}Z_{\sigma,d+i}=C_{\sigma',d+i}Z_{\sigma',d+i}$ on $\mathfrak{M}_{0,\sigma}^{\vee}\medcap\mathfrak{M}_{0,\sigma'}^{\vee}$ for any pair $\sigma,\sigma'$ of sign vectors. Thus, $\bm{Z}_{i}:=C_{\sigma,d+i}Z_{\sigma,d+i}$ are global holomorphic functions on $\mathfrak{M}_0^{\vee}$. 
		
		Let $S_{\ell}$ be the circuits corresponding to the relation $u_{\ell}=\sum_{i=1}^d a_{\ell i}u_i$ for $\ell=d+1,\ldots,n$. We have Kähler parameters $q^{\beta_{S_{\ell}}}$ associated to the primitive curve classes $\beta_{S_{\ell}}$ (see Section \ref{sec:circuits}). Let \[\bm{Z}_{\ell}:=q^{\beta_{S_{\ell}}}\prod_{i=1}^d Z_{d+i}^{u_{\ell}^{(i)}}=q^{\beta_{S_{\ell}}}\prod_{i=1}^d Z_{d+i}^{a_{\ell i}}.\] 
		The generating functions can be expressed locally in term of semi-flat complex coordinates as follows. 
		
		\begin{prop}
		For $j=1,\ldots,d$, denote by $\bm{j}$ be the collection of all $k\in\{1,\ldots,n\}$ such that $u_k^{(j)}\ne 0$. On the open subset $(\pi^\vee)^{-1}(B^0_{\sigma}\cap \mathcal{C}_{\bm{h}})\subset\mathfrak{M}_0^{\vee}$, we have
		\[
		\bm{u}_j=C_{\sigma,j}Z_{\sigma,j}(1+\bm{Z}_{j_1})\ldots(1+\bm{Z}_{j_N}),
		\]
		where $j_1,\ldots,j_N$ are the set of indices such that $h_{j_i}=j$, and
		\[
		\bm{v}_j=\exp\left(\int_{\CP^1_{j}}-\bar{\omega}\right)C_{\sigma,j}^{-1}Z_{\sigma,j}^{-1}\left(\prod_{k\in \bm{j}\setminus \{j_1,\ldots,j_N\}}1+\bm{Z}_{k}\right).
		\]
		where $\CP^1_{j}$ is the holomorphic sphere obtained from gluing $\beta^-_{j}$ and $\beta^+_{j}$ in $H_2(\overline{\mathfrak{M}}_{u,\lambda},L_{b_{\sigma}})$.
		\end{prop}
	    
		\begin{proof}
			This follows from Propositions \ref{prop:cptmaslov2}, \ref{prop:gw2} and Definitions \ref{def:semiflatcoord},\ref{def:gen} (see also \cite[Proposition 4.39]{CLL}).
		\end{proof}

		\subsection{SYZ mirror and its resolution} \label{sec:mirror}
		Set $q^{\beta_{S_{\ell}}}=\exp\left(-\int_{\beta_{S_{\ell}}}\bar{\omega}\right)$. Since the curve classes $\beta_{S_{\ell}}$ are contained in $\mathfrak{M}_{u,\lambda}$, we can rescale $\bar{\omega}$ such that $\exp\left(-\int_{\beta_{S_{\ell}}}\bar{\omega}\right)=\exp\left(-\int_{\beta_{S_{\ell}}}\omega\right)$. By Definition \ref{def:SYZ}, an SYZ mirror is given by $\mathrm{Spec}(R)$ where $R$ is the subring of coordinate ring on $(\pi^{\vee})^{-1}(B^0\setminus\bigcup_i^n W_i)$ generated by the functions $\bm{u}_i$ and $\bm{v}_i$ for $i=1,\ldots,d$.  By combining the above propositions, we obtain the following.
		
		\begin{theorem} \label{thm:SYZmir}
			An SYZ mirror of  $\mathfrak{M}_{u,\lambda}- D^{-}$ is
			\[
			\mathfrak{M}_{u,\lambda}^{\vee}=\left\{((\bm{u}_1,\bm{v}_1,\ldots,\bm{u}_d,\bm{v}_d), (\bm{Z}_{1},\ldots,\bm{Z}_{d}))\in\C^{2d}\times (\C^{\times})^d|\bm{u}_j\bm{v}_j=\prod_{k\in \bm{j}}(1+\bm{Z}_k), j=1,\ldots,d\right\}.
			\]
		\end{theorem}
		For simplicity we have rescaled the variables $\bm{u}_i$ so that the constant terms $\exp\left(-\int_{\CP^1_{j}}\bar{\omega}\right)$ for $j=1,\ldots,d$ do not appear in the above expression.
		\begin{example}
		An SYZ mirror of $T^*\CP^2$ is the subvariety of $\C^4\times (\C^{\times})^2$ given by
			\begin{align*}
			\bm{u}_1\bm{v}_1=(1+\bm{Z}_1)(1+q^{\beta_{S_3}}\bm{Z}_1^{-1}\bm{Z}_2^{-1});\\
			\bm{u}_2\bm{v}_2=(1+\bm{Z}_2)(1+q^{\beta_{S_3}}\bm{Z}_1^{-1}\bm{Z}_2^{-1}).
			\end{align*}
			Note that this subvariety is singular at the one-dimensional loci $\{\bm{Z}_1=-1, \bm{Z}_2=q^{\beta_{S_3}}, \bm{u}_1=\bm{v}_1=0\}$ and $\{\bm{Z}_2=-1, \bm{Z}_1=q^{\beta_{S_3}}, \bm{u}_2=\bm{v}_2=0\}$.
		\end{example}

In general $\mathfrak{M}_{u,\lambda}^{\vee}$ is singular. The wall and chamber structure of the Lagrangian torus fibration explained in Section \ref{sec:chambers} gives a resolution of $\mathfrak{M}_{u,\lambda}^{\vee}$, provided that $\mathfrak{M}_{u,\lambda}$ is smooth. In the following we construct this resolution.  The construction can be justified by Lagrangian Floer theory of immersed Lagrangians which is explained in \cite{HL,HKL18} (see also \cite{seidel97} and \cite{PT17} for more Floer theoretical aspects on gluing the chambers).  We will study more about Lagrangian Floer theory in future work.  In the following we glue up the resolution from local charts by hand.

\textbf{Step 1.} First we glue the charts corresponding to smooth torus fibers by wall-crossing functions.  Recall that we have a collection of tropical hyperplanes which divide the base into chambers (see Figure \ref{fig:chambers} and Section \ref{sec:chambers} for the labels).
For each chamber $\mathcal{C}_{\bm{h}}$, we define a chart $U_{\bm{h}}\cong (\C^{\times})^{d}\times (\C^{\times})^{d}$ by
\[
U_{\bm{h}}=\left\{\left((\bm{u}^{(\bm{h})}_1,\bm{v}^{(\bm{h})}_1,\ldots,\bm{u}^{(\bm{h})}_d,\bm{v}^{(\bm{h})}_d),(\bm{Z}_1,\ldots,\bm{Z}_d)\right)\in (\C^{\times})^{2d}\times (\C^{\times})^{d}\bigr|\bm{u}^{(\bm{h})}_i\bm{v}^{(\bm{h})}_i=1, i=1,\ldots,d\right\}.
\]

Consider a pair of chambers $\mathcal{C}_{\bm{h}}$ and $\mathcal{C}_{\bm{h}'}$ where $\bm{h}=(h_1,\ldots,h_n)$ and $\bm{h}'=(h'_1,\ldots,h'_n)$. For $j=1,\ldots,d$, let $J_{j,\bm{h},\bm{h}'}$ be the set of all indices $k\in\{1,\ldots,n\}$ such that $h_k\ne h'_k$ and either $h_k=j$ or $h'_k=j$.  These indices label the hyperplanes which give walls in between the two chambers involving the $j$-th direction.

Let $U_{\bm{h},\bm{h}'}\subset U_{\bm{h}}$ be the open subset defined by 
\[
U_{\bm{h},\bm{h}'}=\left\{\left((\bm{u}^{(\bm{h})}_1,\bm{v}^{(\bm{h})}_1,\ldots,\bm{u}^{(\bm{h})}_d,\bm{v}^{(\bm{h})}_d),(\bm{Z}_1,\ldots,\bm{Z}_d)\right)\in U_{\bm{h}}\bigr|1+\bm{Z}_k\ne 0  \textrm{ for all } k\in\bigcup_{j=1}^d J_{j,\bm{h},\bm{h}'}\right\}.
\]
Let $\delta^{(\bm{h},\bm{h}')}_j=\delta^{(\bm{h}',\bm{h})}_j=0$ if $J_{j,\bm{h},\bm{h}'}=\emptyset$. Let $\delta^{(\bm{h},\bm{h}')}_j=1$ and $\delta^{(\bm{h}',\bm{h})}_j=0$ if there exist(and hence for all) $k\in J_{j,\bm{h},\bm{h}'}$ such that $h'_k=j$. Let $\delta^{(\bm{h},\bm{h}')}_j=0$ and $\delta^{(\bm{h}',\bm{h})}_j=1$ if there exist $k\in J_{j,\bm{h},\bm{h}'}$ such that $h_k=j$. We glue $U_{\bm{h}}$ and $U_{\bm{h}'}$ via the biholomorphism $\psi_{\bm{h},\bm{h}'}:U_{\bm{h},\bm{h}'}\to U_{\bm{h}',\bm{h}}$ defined by
\[
\bm{u}^{(\bm{h}')}_j=\bm{u}^{(\bm{h})}_j\left(\prod_{k\in J_{j,\bm{h},\bm{h}'}}(1+\bm{Z}_k)^{-1} \right)^{\delta^{(\bm{h},\bm{h}')}_j}\left(\prod_{k\in J_{j,\bm{h},\bm{h}'}}1+\bm{Z}_k\right)^{\delta^{(\bm{h}',\bm{h})}_j};
\]
\[
\bm{v}^{(\bm{h}')}_j=\bm{v}^{(\bm{h})}_j\left(\prod_{k\in J_{j,\bm{h},\bm{h}'}}1+\bm{Z}_k\right)^{\delta^{(\bm{h},\bm{h}')}_j}\left(\prod_{k\in J_{j,\bm{h},\bm{h}'}}(1+\bm{Z}_k)^{-1} \right)^{\delta^{(\bm{h}',\bm{h})}_j}
\]
and the variables $\bm{Z}_i$ are identified trivially.
 
\textbf{Step 2.} We now glue in the charts corresponding to singular SYZ fibers.  For each tropical hyperplane $H\subset\T^d$, we can associate to it its dual simplex $\Delta\subset(\R^d)^*$.  (Note that $\Delta$ can have dimension less than d.) For $k\ge 1$, each $k$-dimensional face $\sigma$ of $\Delta$ corresponds to a $d-k$-dimensional tropical stratum $H_{\sigma}$ of $H$. The $0$-dimensional faces(vertices) of $\Delta$ corresponds to the tropical chambers adjacent to $H$. We will denote by $|\sigma|$ the dimension of $\sigma$, and note that $\sigma$ is itself a simplex. Let $\Delta_1,\ldots,\Delta_n$ be the dual simplexes of the tropical hyperplanes $H_1,\ldots,H_n$.

We will abuse notation and denote by $\mathcal{H}$ both the tropical hyperplane arrangement $\mathcal{H}=\{H_i\}_{i=1}^n$ and the union $\mathcal{H}=\bigcup_{i=1}^n H_i$. We define a stratification of $\mathcal{H}$ as follows. Let $\bm{\sigma}=\{\sigma_{j_1},\ldots,\sigma_{j_{\nu}}\}$ be a collection such that $\sigma_{j_i}$ is a face of $\Delta_{j_i}$, and set $|\bm{\sigma}|:=\sum_{\sigma_j\in\bm{\sigma}}|\sigma_j|$. Let $\mathcal{H}_{\bm{\sigma}}\subset\mathcal{H}$ be the set
\[
\mathcal{H}_{\bm{\sigma}}=\bigcap_{\sigma_j\in\bm{\sigma}} (H_j)_{\sigma_j}.
\]
We define the $0$-dimensional strata of $\mathcal{H}$ to be points of the form $\mathcal{H}_{\bm{\sigma}}$ where $|\bm{\sigma}|=d$. We then define the $\ell$-dimensional strata of $\mathcal{H}$ for for $\ell\ge 1$ to be the connected components of $\mathcal{H}_{\bm{\sigma}}\setminus\mathcal{H}_{\ell-1}$, where $|\bm{\sigma}|=d-\ell$, and $\mathcal{H}_{\ell-1}$ denotes the union of the $(\ell-1)$-dimensional strata $\Theta$ of $\mathcal{H}$.

Let $\{e_1,\ldots,e_d\}$ be the standard basis on $(\R^d)^*$, and denote by $<,>$ the standard pairing between $\R^d$ and $(\R^d)^*$. For each $\bm{\sigma}$ with $1\le|\bm{\sigma}|\le d$, we associate to it a collection of primitive and linearly independent vectors $\{\vec{a}^{\bm{\sigma}}_1,\ldots,\vec{a}^{\bm{\sigma}}_{\ell}\}$ parallel to $\mathcal{H}_{\bm{\sigma}}$. For each $\sigma_{j_k}\in\bm{\sigma}$, we associate to it a collection of primitive vectors $\{\vec{a}^{\sigma_{j_k}}_1,\ldots,\vec{a}^{\sigma_{j_k}}_{|\sigma_{j_k}|+1}\}$ such that $\vec{a}^{\sigma_{j_k}}_i$ is normal to the $i^{\mathrm{th}}$ facet of $\sigma_{j_k}$(notice that the number of facets of $\sigma_{j_k}$ is $|\sigma_{j_k}|+1)$, and parallel to $\bigcap_{\sigma_j\in\bm{\sigma},j\ne j_k} (H_j)_{\sigma_j}$. In particular, we can choose $\{\vec{a}^{\sigma_{j_k}}_1,\ldots,\vec{a}^{\sigma_{j_k}}_{|\sigma_{j_k}|+1}\}$ such that $\vec{a}^{\sigma_{j_k}}_{|\sigma_{j_k}|+1}=-\sum_{i=1}^{|\sigma_{j_k}|} \vec{a}^{\sigma_{j_k}}_{i}$.

For each $\ell$-dimensional stratum $\Theta$, there exists a unique collection $\bm{\sigma}=\{\sigma_{j_1},\ldots,\sigma_{j_{\nu}}\}$ such that $|\bm{\sigma}|=d-\ell$, and $\Theta\subset\mathcal{H}_{\bm{\sigma}}$. We will call a stratum $\Theta$ \textit{admissible} if $\bigcap_{k=1}^\nu H_{\R,j_k}\ne\emptyset$.  

Now, we associate to each admissible stratum $\Theta$ a chart $U_{\Theta}$ defined by
\[
U_{\Theta}=\left\{\begin{array}{l}\left((x^{(\Theta,\sigma_{j_i})}_{1},\ldots,x^{(\Theta,\sigma_{j_i})}_{|\sigma_{j_i}|+1} )_{i=1,\ldots,\nu},(y^{(\Theta)}_1,\ldots,y^{(\Theta)}_{\ell}),(\bm{Z}_1,\ldots,\bm{Z}_d)\right)\in\left(\prod_{i=1}^{\nu}\C^{|\sigma_{j_i}|+1}\right)\times(\C^{\times})^{\ell}\times (\C^{\times})^d\\
\text{ s.t. } \prod_{k=1}^{|\sigma_{j_i}|+1}x^{(\Theta,\sigma_{j_i})}_{k}=1+\bm{Z}_{j_i}, i=1,\ldots,\nu
\end{array}\right\}.
\]
We glue $U_{\Theta}$ to the resulting space from Step 1. For any tropical chamber $\mathcal{C}_{\bm{h}}$ adjacent to $\Theta$, we define an open embedding $\psi_{\bm{h},\Theta}:U_{\bm{h}}\to U_{\Theta}$ by
\[
 y^{(\Theta)}_k=\prod_{i=1}^{d}(\bm{u}^{(\bm{h})}_i)^{<e_i,\vec{a}^{\bm{\sigma}}_k>}, \quad k=1,\ldots,\ell;
\]
\[
x^{(\Theta,\sigma_{j_i})}_{k}=\prod_{m=1}^d (\bm{u}^{(\bm{h})}_m)^{<e_m,\vec{a}^{\sigma_{j_i}}_k>} \qquad i=1,\ldots,\nu;
\]
if the $k^{\mathrm{th}}$ facet of $\sigma_{j_i}$ is adjacent to the vertex of $\sigma_{j_i}$ corresponding to the chamber $\mathcal{C}_{\bm{h}}$, and
\[
x^{(\Theta,\sigma_{j_i})}_{k}=(1+\bm{Z}_{j_i})\prod_{m=1}^d (\bm{u}^{(\bm{h})}_m)^{<e_m,\vec{a}^{\sigma_{j_i}}_k>} \qquad i=1,\ldots,\nu,
\]
otherwise. The variables $\bm{Z}_i$ are identified trivially.

We denote the smooth variety obtained from this gluing by $\widetilde{\mathfrak{M}_{u,\lambda}^{\vee}}$.  We have $H^0(\widetilde{\mathfrak{M}_{u,\lambda}^{\vee}},\mathcal{O}_{\widetilde{\mathfrak{M}_{u,\lambda}^{\vee}}})=R$. Thus, the resolution $\widetilde{\mathfrak{M}_{u,\lambda}^{\vee}}\to \mathfrak{M}_{u,\lambda}^{\vee}$ is the affinization map.

Figure \ref{fig:hypertoric-gluing} shows an example of the above gluing procedure.

\begin{remark}
$\widetilde{\mathfrak{M}_{u,\lambda}^{\vee}}$ is equipped with a holomorphic symplectic form $\sum_{i=1}^d  d\log\bu_i\wedge d\log\bm{Z}_i$ and a holomorphic volume form
\[
\Omega^{\vee}=\prod_{i=1}^d d\log\bu_i\wedge d\log\bm{Z}_i
\]
which is preserved under the change of coordinates (up to signs).
Let $\gamma\in H_{2d}(\widetilde{\mathfrak{M}_{u,\lambda}^{\vee}})$ and consider the period integrals
\begin{equation}
\label{period}
\int_{\gamma}\Omega^{\vee}.
\end{equation}
Let $\bm{\pi}=(\bm{Z_1},\ldots,\bm{Z}_d):\widetilde{\mathfrak{M}_{u,\lambda}^{\vee}}\to (\C^{\times})^d$ be projection map. $\bm{\pi}$ is a $(\C^{\times})^d$-fibration over base $(\C^{\times})^d$. Denote by $\bm{H}$ the union of multiplicative hyperplanes
\[
\bm{H}=\bigcup_{i=1}^n \{\bm{Z}\in(\C^{\times})^d|q_i\bm{Z}^{\hat{\lambda}_{\R,i}}_i=-1\},
\]
where $q_i=1$, for $i=1,\ldots,d$ and $q_{\ell}=q^{\beta_{S_{\ell}}}$, for $\ell=d+1,\ldots,n$. The period integrals (\ref{period}) reduces to integration over relative cycles $\gamma'\in H_d((\C^{\times})^d,\bm{H})$, 
\[
\int_{\gamma'}\prod_{i=1}^d d\log\bm{Z}_i
\]
by dimension reduction similar to the case of toric varieties (see e.g. \cite{CLT11}). In \cite{MS}, Mcbreen and Shenfeld observed that certain period integrals on $(\C^{\times})^d$ with local coefficients satisfy the same GKZ system as the $T^d\times\C^{\times}$-equivariant quantum cohomology.
\end{remark}

\begin{figure}[htb!]
	\includegraphics[scale=0.65]{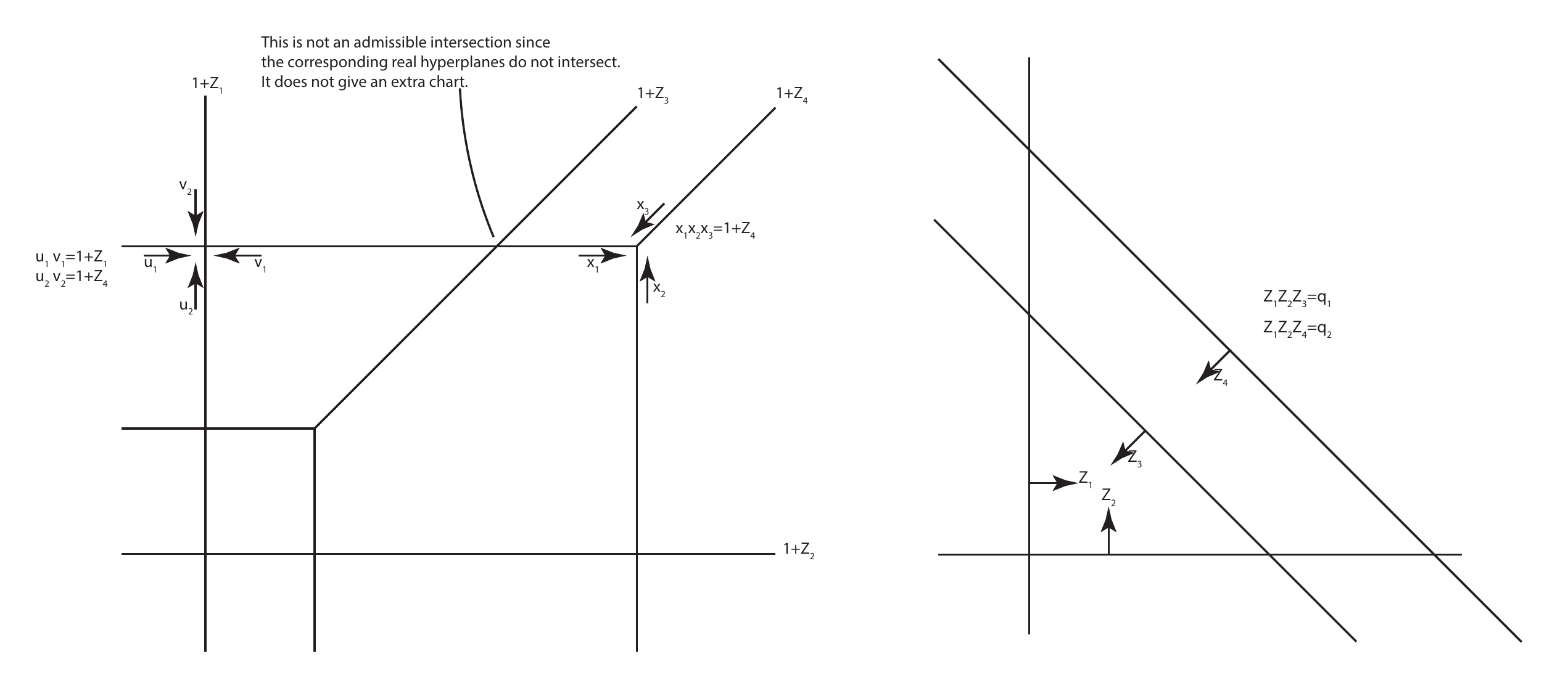}
	\caption{An example of gluing.  The left is the tropical hyperplane arrangement, and the right is the real hyperplane arrangement.  Each admissible intersection stratum gives a local chart and they are glued to adjacent chambers according to the vectors associated to the variables.}
	\label{fig:hypertoric-gluing}
\end{figure}

		\subsection{Multiplicative hypertoric varieties} \label{sec:multiplicative}
		In this subsection, we show that
		smooth multiplicative hypertoric varieties (see e.g.\cite{ganev14})  provide alternative resolutions to our SYZ mirrors.
		
		Let's first review the construction of multiplicative hypertoric varieties. Let 
		\[
		(T^*\mathbb{C}^{n})^{\circ}=\{(z,w)\in T^*\mathbb{C}^n|1+z_iw_i\ne 0, i=1,\ldots,n\}.
		\]
		We equip $(T^*\mathbb{C}^{n})^{\circ}$ with the holomorphic symplectic form
		\[
		\omega^{\circ}=\sum_i^n \cfrac{dz_i\wedge dw_i}{1+z_iw_i}.
		\] 
		Let $\vec{t}=(t_1,\ldots,t_n)\in (\mathbb{C}^{\times})^n$ act on $(T^*\mathbb{C}^{n})^{\circ}$ by
		\[
		\vec{t}\cdot(z,w)=(t_1z_1,t_1^{-1}w_1,\ldots,t_nz_n,t_n^{-1}w_n).
		\]
		This action comes with a $(\mathbb{C}^{\times})^n$-valued moment map(for the general theory of Lie-group valued moment maps, see \cite{AMM}) $\tilde{\bm{\mu}}:(T^*\mathbb{C}^{n})^{\circ}\to (\mathbb{C}^{\times})^n$ given by
		\[
		\tilde{\bm{\mu}}(z,w)=\left((1+z_1w_1),\ldots,(1+z_nw_n)\right).
		\]
	    Let $K_{\C}\subset (\mathbb{C}^{\times})^n$ be the subtorus defined by the collection of vectors $u$ as in Section \ref{sec:hypertoric}. Let $(\iota^*_{ij})_{1\le i\le n-d, 1\le j\le n}$ be the matrix associated to $\iota^*:(\mathfrak{t}^n)^*\to \mathfrak{k}^*$.
		The multiplicative moment map $\bm{\mu}:(T^*\mathbb{C}^{n})^{\circ}\to K_{\C}$ of the $(\mathbb{C}^{\times})^n$-action on  $(T^*\mathbb{C}^{n})^{\circ}$ restricted to $K_{\C}$ is given by
		\[
		\bm{\mu}(z,w)=\left(\prod_{j=1}^n (1+z_jw_j)^{\iota^*_{1j}},\ldots,\prod_{j=1}^n (1+z_jw_j)^{\iota^*_{(n-d)j}}\right).
		\]
		Let $\eta=(\eta_1,\ldots,\eta_{n-d})\in K_{\C}$, and let $\chi:K_{\C}\to \C^{\times}$ be a character. We define a \textit{multiplicative hypertoric variety} to be the GIT quotient 
		\[
		X_{u,\chi,\eta}=\bm{\mu}^{-1}(\eta)//_{\chi} K_{\C}.
		\]
		or equivalently,
		\[
		X_{u,\chi,\eta}=\mathrm{Proj}\left(\bigoplus_{k\ge 0} \mathcal{O}\left(\bm{\mu}^{-1}(\eta)\right)^{\chi^k}\right).
		\]
		
		Set   $q=\left((-1)^{\sigma_{d+1}+1}q^{\beta_{S_{d+1}}},\ldots,(-1)^{\sigma_{n}+1}q^{\beta_{S_n}}\right)\in K_{\C}$, where $\sigma_{\ell}$ is the parity of $\sum_{i=1}^d a_{\ell i}$, and $a_{\ell i}$ are coefficients in $u_{\ell}=\sum_{i=1}^d a_{\ell i}u_i$. Consider the multiplicative hypertoric variety $X_{u,0,q}$. We have
		\[
		X_{u,0,q}=\mathrm{Spec}\left(\C[\bm{\mu}^{-1}(q)]^{K_{\C}}\right),
		\] 
		where $\C[\bm{\mu}^{-1}(q)]^{K_{\C}}$ denotes the $K_{\C}$-invariant subring of $\C[\bm{\mu}^{-1}(q)]$.
		
		Let $\Pi=(\pi^*_{ji})_{1\le j\le n, 1\le i\le d}$ be the matrix associated to the map $\pi^*:(\mathfrak{t}^d)^*\to (\mathfrak{t}^n)^*$ with respect to the ordered basis $u_1,\ldots,u_d$. $(\pi^*_{ji})_{1\le j,i\le d}$ is the identity $d$ by $d$ matrix. Since $\Pi$ is \textit{totally unimodular}, the remaining entries take values in $\{-1,0,1\}$. The columns of $\Pi$ correspond to $K_{\C}$-invariant polynomials $\bm{z}_i=\prod_{j=1}^n x_{ij}^{|\pi^*_{ji}|}$ and $\bm{w}_i=\prod_{j=1}^n y_{ij}^{|\pi^*_{ji}|}$, where $x_{ij}=z_j$, $y_{ij}=w_j$ if $\pi^*_{ji}\ge 0$, and $x_{ij}=w_j$, $y_{ij}=z_j$ if $\pi^*_{ji}<0$. Denote by $S$ the multiplicative system generated by $\bm{z}_i,\bm{w}_i$, and $z_iw_i$ for $i=1,\ldots,d$.
		
		\begin{lemma}
		\label{lemma:gen}
        $S^{-1}\C[\bm{\mu}^{-1}(q)]^{K_{\C}}$ is generated by $\bm{z}_i^{\pm 1},\bm{w}_i^{\pm 1}$, and $(z_iw_i)^{\pm 1}$ for $i=1,\ldots,d$.
		\end{lemma}
	
	    \begin{proof}
	    Let $f=\prod_{i=1}^n z_i^{a_i}\prod_{i=1}^n w_i^{b_i}$ be an arbitrary nonconstant Laurent monomial in $S^{-1}\C[\bm{\mu}^{-1}(q)]$. If $f$ is not divisible by neither $\bm{z}_i^{\pm 1}$ nor $\bm{w}_i^{\pm 1}$ for $i=1,\ldots,d$, then the vector $\langle a_1-b_1,\ldots,a_n-b_n\rangle$ is not in the kernel of $\iota^*:(\mathfrak{t}^n)^*\to \mathfrak{k}^*$ unless it is the zero vector. In the first case, $f$ is not $K_{\C}$-invariant, while in the second case, $f$ is a product of $(z_iw_i)^{\pm 1}$.
	    \end{proof}
	
		\begin{prop}
		\label{prop:birational}
		For a generic choice of $\chi$, $X_{u,\chi,q}$ is a resolution of $\mathfrak{M}_{u,\lambda}^{\vee}$. 
		\end{prop}
	
	    \begin{proof}	
	    We have a ring homomorphism $\varphi:R\to\C[\bm{\mu}^{-1}(q)]^{K_{\C}}$ given by
	    \[
	    \varphi(\bm{u}_i)=(-1)^{\mathrm{sgn}\left(\sum_{j=1}^n |\pi^*_{ji}|\right)}\bm{z}_i,\quad \varphi(\bm{v}_i)=\bm{w}_i,\quad \varphi(\bm{Z}_{i})=-1-z_iw_i,\text{ for } i=1,\ldots,d.
	    \]
	    Denote by $R'$ the ring obtained by localizing $R$ at the multiplicative system generated by $\bm{u}_i,\bm{v}_i$, and $1+\bm{Z}_i$ for $i=1,\ldots,d$. The induced map $\varphi_*:X_{u,0,q}\to\mathfrak{M}_{u,\lambda}^{\vee}$ is birational since $\varphi$ descends to a ring isomorphism $R'\cong S^{-1}\C[\bm{\mu}^{-1}(q)]^{K_{\C}}$. When the Kähler parameters of $\mathfrak{M}_{u,\lambda}$ are generic (i.e. $\mathcal{H}_{\R}$ is simple), $X_{u,\chi,q}$ is smooth and is independent of $\chi$, and therefore we have a resolution of $\mathfrak{M}_{u,\lambda}^{\vee}$ by $X_{u,\chi,q}$. On the other hand, if the Kähler parameters are not generic, $X_{u,0,q}$ is singular. However, the affinization map $X_{u,\chi,q}\to X_{u,0,q}=\Spec{H^0(X_{u,\chi,q},\mathcal{O}_{X_{u,\chi,q}})}$ is a resolution. In this case, the composition $X_{u,\chi,q}\to X_{u,0,q}\to\mathfrak{M}_{u,\lambda}^{\vee}$ is a resolution.	    
	    \end{proof}

		\bibliographystyle{amsalpha}
		\bibliography{geometry}
	\end{document}